\documentclass{amsart}

\usepackage[table]{xcolor}
\definecolor{lightgray}{gray}{0.9}

\usepackage{amsmath, amsfonts, amssymb, amsthm}
\usepackage{tikz-cd}
\usepackage[utf8]{inputenc}
\usepackage{csquotes}
\usepackage{subcaption}
\usepackage[hidelinks]{hyperref}
\usepackage[english]{babel}
\usepackage{pinlabel}

\usepackage[giveninits, citestyle=numeric-comp,sorting=nty,sortcites, backend=biber]{biblatex}

\AtEveryBibitem{
  \clearfield{doi}
  \clearfield{url}
  \clearfield{month}
  \clearfield{day}
  \clearfield{issn}
  \clearfield{pages}
  \clearfield{publisher}
  \clearfield{isbn}
  \clearfield{pagetotal}
}
\DeclareFieldFormat*{urldate}{}
\DeclareFieldFormat*{title}{#1}
\renewbibmacro{in:}{}
\DeclareFieldFormat*{volume}{vol. #1}
\DeclareFieldFormat*{number}{no. #1}
\renewbibmacro*{volume+number+eid}{%
  \printfield{volume}%
  \setunit*{\addcomma\space}
  \printfield{number}%
  \setunit{\addcomma\space}%
  \printfield{eid}
}
\DeclareFieldFormat[misc]{date}{(#1)}

\newtheorem{theorem}{Theorem}[section]
\newtheorem{proposition}[theorem]{Proposition}
\newtheorem{corollary}[theorem]{Corollary}
\newtheorem{lemma}[theorem]{Lemma}
\newtheorem{definition}[theorem]{Definition}
\newtheorem{remark}[theorem]{Remark}
\newtheorem{example}[theorem]{Example}

\newcommand{\integers}{\mathbb{Z}}

\newcommand{\ihat}{\hat{\iota}}

\newcommand{\highlight}[1]{#1}

 \DeclareMathOperator{\mcg}{Map}
 
 \DeclareMathOperator{\orb}{Orb}

\title{Braided multitwists}
\author{Rodrigo De Pool}
\thanks{The author acknowledges financial support from the grant  CEX2019-000904-S funded by  MCIN/AEI/ 10.13039/501100011033 and from the grant  PGC2018-101179-B-I00.}

\begin{document}

\maketitle

\begin{abstract}
  We provide a characterization for multitwists satisfying the braid relation in the mapping class group of an orientable surface.
\end{abstract}


\section{Introduction}\label{introd}

Let $S$ be an orientable surface  \highlight{possibly with punctures or boundary}. Let $\mcg(S)$ be the mapping class group of $S$, that is, the group of homeomorphisms of $S$ up to homotopy. In this article we characterize \emph{braided multitwists}, i.e, multitwists $\tau_A, \tau_B \in \mcg(S)$ satisfying the braid relation $\tau_A \tau_B \tau_A = \tau_B \tau_A \tau_B$. First, let us consider some examples:

\begin{example}
Consider $a$ and $b$ two curves in $S$ intersecting once. Denote their Dehn twists by $\delta_a$ and $\delta_b$. It is known that \[ \delta_a \delta_b \delta_a = \delta_b \delta_a \delta_b, \] and by the same equality the inverses $\delta_a^{-1}, \delta_b^{-1}$  are also braided. It is a lemma that two Dehn twists $\delta_a,\delta_b$ are braided if and only if the curves $a$ and $b$ intersect once (\cite[Lemma 4.3]{mccarthy_automorphisms_1986}).
\end{example}

Now, we construct braided multitwists that are not Dehn twists.

\begin{example}\label{ex1}

Let $\tau_A = \delta_{a_1} \delta_{a_2}\delta_{a_3}\delta_{a_4}^{-1}$ and $\tau_B = \delta_{b_1}\delta_{b_2}\delta_{b_3}\delta_{b_4}^{-1}$ be two multitwists along the curves $a_i,b_j$ shown in Figure \ref{fig:img1}. Since $\delta_{a_i}$ commutes with $\delta_{b_j}$ whenever $i\neq j$ and   $a_i, b_i$ intersect once, we  deduce  $\tau_A \tau_B \tau_A = \tau_B \tau_A \tau_B$ from the previous example. 

\end{example}

Note that in Example \ref{ex1} the multitwists $\tau_A$ and $\tau_B$ do not share any common Dehn twist.  An easy way  to create new braided multitwists is to add common components with the  same power. We do so in the next example:

\begin{example}\label{ex:crucial}
Let $\tau_A,\tau_B$ be the multitwists of Example \ref{ex1} and let $d$ be the curve shown in Figure \ref{fig:img1}. Then, the multitwists $ \tau_A \delta_d^m$ and $ \tau_B \delta_d^m$ satisfy the braid relation. 
\end{example}

\begin{figure}
	    \labellist
		\small\hair 2pt
		\pinlabel $a_2$  at 268 365
		\pinlabel {$b_2$}  at 280 295
		\pinlabel {$a_4$}  at 70 194
		\pinlabel {$b_4$}  at 185 232
		\pinlabel {$d$}  at 280 257
		\pinlabel {$b_1$}  at 362 249
		\pinlabel {$a_1$}  at 517 210
		\pinlabel {$b_3$}  at 290 145
		\pinlabel {$a_3$}  at 285 65
		\endlabellist
		\centering
		\includegraphics[width=0.5\linewidth]{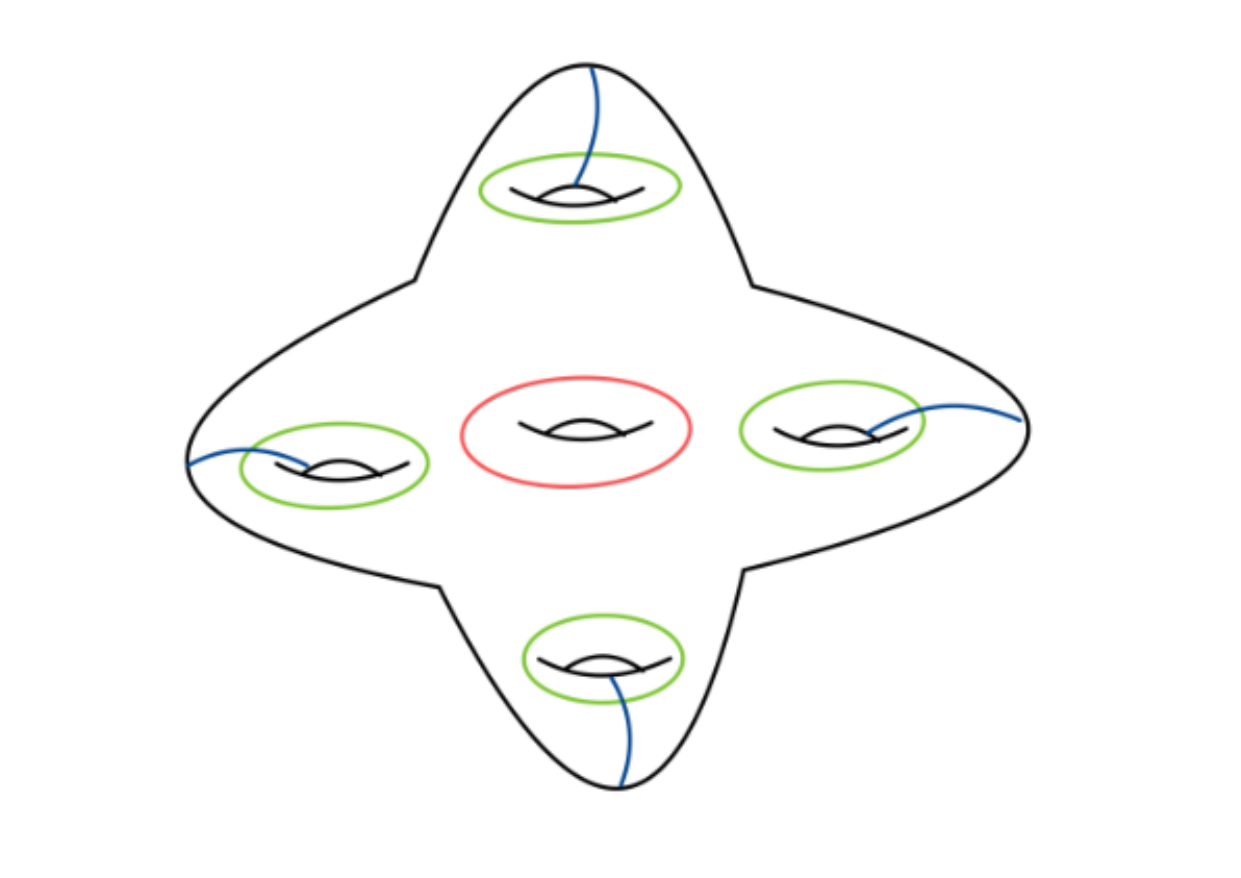}
		\caption{Genus five surface.}
		\label{fig:img1}
\end{figure}
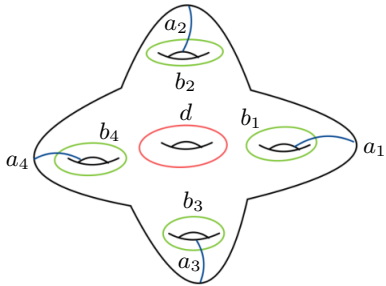

The goal of this article is to prove that braided multitwists are essentially as in Example \ref{ex:crucial}.

\begin{theorem}\label{thm1}
	Let $S$ be an orientable surface. If $\tau_A, \tau_B \in \mcg(S)$ are two braided multitwists, then
	\begin{align*}
		& \tau_A=\delta_{a_1}^{n_1} \dots \delta_{a_k}^{n_k} \tau_C\\ 
		& \tau_B=\delta_{b_1}^{n_1} \dots \delta_{b_k}^{n_k} \tau_C,
	\end{align*}
	 where $\tau_C$ is a common multitwist,  $n_i \in \{-1,1\}$ and the curves $a_1,\dots,a_k,\,b_1,\dots,b_k$ are pairwise disjoint except for $i(a_i,b_i)=1$.
\end{theorem}

As an application of Theorem \ref{thm1}, we can describe homomorphisms from braid groups to mapping class groups that send standard generators to multitwists. Recall  the braid group on $n$-strands is given by the presentation  \[ B_n = \langle \sigma_1, \,\dots, \, \sigma_{n-1}| \: \sigma_i \sigma_{i+1} \sigma_i  = \sigma_{i+1}  \sigma_i  \sigma_{i+1} \;\forall i, \;\; [\sigma_i, \sigma_j]=1 \;\forall |i-j|>1 \rangle, \]  and the $\sigma_i$'s  are said to be the \emph{standard braid generators}. 

A natural way to produce an embedding $\varphi: B_n \hookrightarrow \mcg(S)$ is to send  each standard generator $\sigma_i$ to a Dehn twist $\delta_{c_i}$, where  the curves $c_i,c_j$ intersect once if $|i-j|=1$ and are disjoint otherwise.  Any such $\varphi$ (or $\varphi^{-1}$) is known as a \emph{geometric} embedding. If we replace $\delta_{c_i}$ by a multitwist, then Theorem \ref{thm1} yields a description of $\varphi$: 

\begin{corollary}\label{cor:braids}
	Let $S$ be an orientable surface and let $B_n$ be the braid group on $n$-strands. If $\varphi: B_n\to \mcg(S)$ sends a standard generator to a multitwist, then $\varphi$ factors as $ \varphi = \left(\prod_{i=1}^k \varphi_i\right) \circ d,$ where $d: B_n \to \prod_{i=1}^k B_n$ is the diagonal homomorphism, $\varphi_i$ is a geometric embedding for $1\leq i<k$, and  $\varphi_k$ has cyclic image.
\end{corollary}

One could hope that every embedding $B_n \hookrightarrow \mcg(S)$ comes from a diagonal decomposition of geometric embeddings; however this is not the case. For instance, Szepietowski constructed in \cite{szepietowski_embedding_2010}  an embedding $B_n \hookrightarrow \mcg(S)$ sending a standard generator to a non-trivial root of a Dehn twist.

Theorem \ref{thm1} fits into the bigger picture of results studying relations between multitwists. A closely related result is \highlight{\cite[Theorem 3.4]{hamidi-tehrani_groups_2002} }  of Hamidi-Tehrani, which  provides sufficient conditions  for two \emph{positive} multitwists to generate a free group of rank two. Leininger improved this result \highlight{in \cite[Theorem 6.1]{leininger_groups_2004} }  by classifying pairs of positive multitwists generating a free group. We highlight that Theorem \ref{thm1} for positive multitwists can be proved using Leininger's classification. However, Leininger's techniques do require the multitwists to be positive and it is not clear how to generalize them to the case of arbitrary multitwists. Thus, our proof of Theorem \ref{thm1} uses a different approach.

The original motivation for Theorem \ref{thm1} is to further understand homomorphisms between pure mapping class groups. We will study this application in a forthcoming paper. 

\subsection*{Acknowledgments}
The author would like to thank his supervisor Javier Aramayona for helpful conversations and the guidance provided.

\section{Preliminaries}\label{sec:preliminares}
Let $S$ be a connected orientable surface. By a curve on $S$ we will mean the homotopy class of an unoriented simple closed curve that does not bound a disk or a punctured disk. Brackets around an oriented curve will be used to denote the homology class of the curve.

For two curves $a$ and $b$  the \emph{(geometric) intersection number} $i(a,b)$ is minimum number of intersection points between representatives of $a$ and $b$. When the curves  $a$ and $b$ are oriented,  the \emph{algebraic intersection number} $\ihat([a],\,[b])$  is the sum of the signs at each intersection point. Importantly, this sum remains invariant regardless of the choice of representatives for  $[a]$ and $[b]$.

Throughout the article, we will blur the difference between curves and their representatives. Additionally, we will often consider representatives that intersect pairwise minimally. It is a theorem that such representatives always exist (see \cite[Chapter 1.2]{farb_primer_2012}).

If $a$ is a curve on $S$ and $A$ is a regular neighbourhood of $a$ given in (oriented) coordinates by $\{(h, \theta):\: h\in[0,1], \;\theta \in \mathbb{R}/(x\sim x+2\pi)\}$, we can define the twisting map $\delta_a:S\to S$ that sends $(h,\, \theta) \mapsto (h, \, \theta + h\theta)$  on $A$ and is the identity elsewhere. The resulting mapping class $\delta_a\in \mcg(S)$ is the \emph{Dehn twist} along the curve $a$.  Note that as a matter of convention, we take our Dehn twists to be twists to the \emph{right}.

A \emph{multitwist} $\tau_A\in \mcg(S)$ is a finite product of Dehn twists 
\[\tau_A = \delta_{a_1}^{n_1} \dots \delta_{a_k}^{n_k},\] where the curves $a_1,\dots, a_k$ are distinct, pairwise disjoint and $n_i \in \mathbb{Z}$. The multitwist is \emph{positive} if $n_i\geq 0$ for every $i$ and it is \emph{negative} if $n_i \leq 0$ for every $i$. The subindex $A$ will be used to denote the set  $A=\{a_1,\dots,a_k\}$, and we call $A$ the set of \emph{curves} of $\tau_A$. Abusing notation, we will refer to $n_i$ as the power of $a_i$ in $\tau_A$.

The mapping class group $\mcg(S)$ acts on the set of curves of $S$ and the action preserves the intersection number. We will write $h\cdot a$ to denote the action of the mapping class $h$ on the curve $a$. Similarly, $\mcg(S)$ acts linearly on the homology of $S$ and preserves the algebraic intersection number. We will write $h\cdot [a]$ for the action of the mapping class $h$ on the homology class $[a]$.

The rest of this section is devoted to studying formulas for the geometric intersection number and algebraic intersection number. For other fundamental results and definitions on mapping class groups, we refer the reader to \cite{farb_primer_2012}.

\subsection{Geometric intersection formulas}
Let $a$, $b$ and $c_j$ be curves on $S$ and consider a positive multitwist $\tau_C  = \delta_{c_1}^{n_1} \dots \delta_{c_k}^{n_k}$. The following is a well-known bound for the intersection number 
\begin{equation}\label{formulaPositiveTwists}
	i(a,b) \geq \left| i(a,\,\tau_C\cdot b) - \sum_{i=1}^k |n_i|\cdot i(a, c_i)\cdot i(b, c_i) \right|.
\end{equation}
Naturally, the same bound applies for negative multitwists. For a proof of this inequality see \cite[Proposition 3.4]{farb_primer_2012}.

The  formula above requires the multitwist to be either positive or negative. However, Theorem \ref{thm1} deals with multitwists that (possibly) have mixed signs. The next bound on intersection number was proved by Ivanov for general multitwists  (see  \cite[Proposition 4.2]{ivanov_subgroups_1992}).

\begin{proposition}\label{prop:formulaIvanov}
Consider two curves $a,b$ and the multitwist $\tau_C = \delta_{c_1}^{n_1} \dots \delta_{c_k}^{n_k}$ with $n_i\in\mathbb{Z}$. Then, 
\[
i(a, b)\geq -i(a, \tau_C(b)) +\sum_{j=1}^k \widetilde{n}_j \cdot i(a,c_j)\cdot i(b, c_j),
\]
where $\widetilde{n}_j = \max{\{|n_j|-2, \,0\}}$.
\end{proposition}

The proof given by Ivanov hides another formula, which will be extensively used in this article. Before introducing this formula, we need one more piece of notation.

Let $a$ be a curve and $\tau_C= \delta_{c_1}^{n_1} \dots \delta_{c_k}^{n_k}$ a multitwist. The intersection points of $a$ with curves in $C$ define a partition $a=a_0 \cup a_1 \cup \dots \cup a_m$, where each $a_i$ is an arc between consecutive intersection points. Since we have an arc for each intersection point, there is a total of $m+1=\sum_{j=1}^k i(a, \,c_j)$ arcs. Now, for each  $a_i$  set $x_i=1$  if the subarc $a_i$ intersects  $c_k, c_l \in C$ and $n_k \cdot n_l > 0$. If  $n_k \cdot n_l < 0$,  set $x_i=0$. We define the function
\[ 
X(a, \,\tau_C) = \sum_{i=1}^m x_i.
\]
The function  $X(a, \, \tau_C)$  \highlight{is non-negative and it } does not depend on the choice of (minimally intersecting) representatives. This is a consequence of the next proposition.

Although Ivanov does not explicitly state the following result, it follows immediately from the first part of his proof of \cite[Proposition 4.2]{ivanov_subgroups_1992}. 

\begin{proposition}[Ivanov's  hidden formula]\label{prop:formulaEscondidaIvanov}
    Consider a curve $a$ and the multitwist $\tau_C = \delta_{c_1}^{n_1} \dots \delta_{c_k}^{n_k}$ with $n_i\in\mathbb{Z}$. Then,
\[
i(a, \tau_C\cdot a) = \sum_{j=1}^k (|n_j|i(a, c_j)-1)i(a, c_j)+X(a, \tau_C).
\]
\end{proposition}

The next example illustrates how to compute $X(a,\,\, \tau_C)$ and $i(a, \tau_C\cdot a) $ using Proposition \ref{prop:formulaEscondidaIvanov}.

\begin{example}\label{ex2}
	Consider the curves  in Figure \ref{fig:ejemplox} and let $\tau_C=\delta_{c_1}^2\delta_{c_2}^{-1}\delta_{c_3}$. We  compute $X(a,\,\, \tau_C)$ and $i(a, \tau_C\cdot a) $.
    
    The intersection points induce the partition  $a=a_0\cup a_1 \cup a_2 \cup a_3$. Since $a_0$ intersects $c_1, \,c_2$ and they have powers $2\cdot -1<0$,  then $x_0=0$. Since $a_1$ intersects only $c_1$ and $2\cdot 2>0$, then $x_1=1$. Similarly, we compute $x_2=1$ and $x_3=0$. It follows that $X(a,\,\,\tau_C) = \sum_{i=0}^3 x_i = 2$ and we obtain $i(a, \tau_C \cdot a) = 8$ using Proposition  \ref{prop:formulaEscondidaIvanov}.

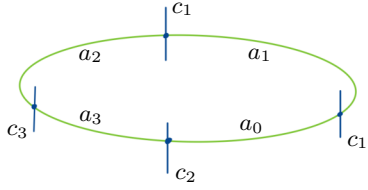
\begin{figure}[h]
\labellist
\small\hair 2pt
\pinlabel {$c_1$} [ ] at 300 340
\pinlabel {$a_2$} [ ] at 155 266
\pinlabel {$a_1$} [ ] at 421 268
\pinlabel {$c_3$} [ ] at 41 148
\pinlabel {$a_3$} [ ] at 156 170
\pinlabel {$c_2$} [ ] at 305 78
\pinlabel {$a_0$} [ ] at 408 155
\pinlabel {$c_1$} [ ] at 575 131
\endlabellist
\centering       
\includegraphics[width=0.4\linewidth]{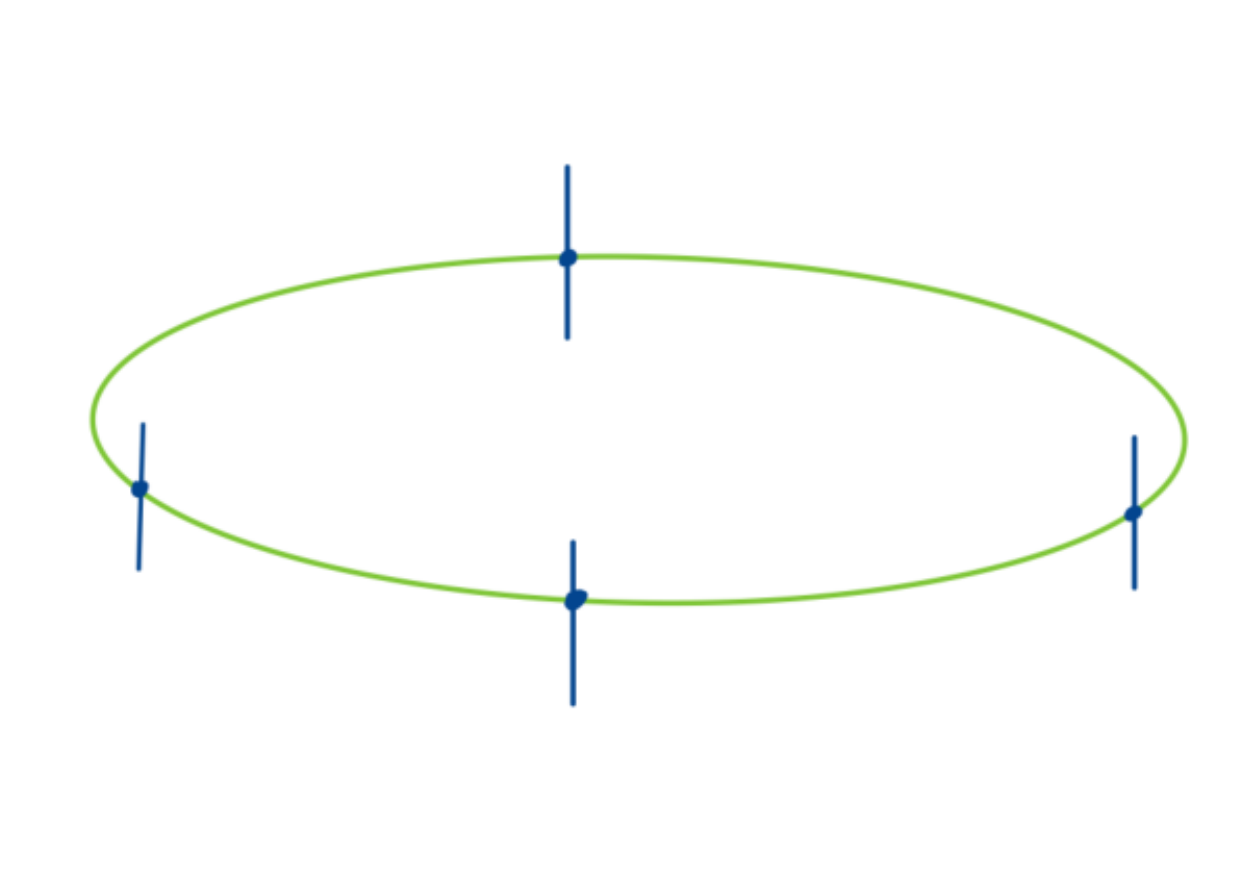}
\caption{Curve $a=\bigcup_{i=0}^3 a_i$ and the intersecting curves $c_1,\,c_2,\,c_3$.}
\label{fig:ejemplox}
\end{figure}

\end{example}

\subsection{Algebraic intersection formulas}

The action of a Dehn twist on the homology of a curve is described in \cite[Proposition 6.3]{farb_primer_2012} by the  formula
\[ \delta_c^n\cdot [a]  = [a]+ n\cdot \ihat([a], \,[c])\cdot [c].\]Recall that $\mcg(S)$ acts linearly on the homology of $S$, thus we may iteratively use this equality  to obtain a formula for the action of a multitwist. Let $\tau_C = \delta_{c_1}^{n_1} \dots \delta_{c_k}^{n_k}$ be a multitwist and $[a]$ the homology class of an oriented curve, it follows that
\begin{equation}\label{homologicalActionMultitwist}
	\tau_C \cdot [a] = [a]+\sum_{i=1}^kn_k\cdot \ihat([a], \,[c_i])\cdot[c_i].   
\end{equation}
From Equation \ref{homologicalActionMultitwist} and the bilinearity of  $\ihat(\cdot, \cdot)$, we derive the following result:

\begin{lemma}\label{lemaInterseccionAlg}
	Consider $a$ and $b$ two oriented curves on $S$.  Let $\tau_C = \delta_{c_1}^{n_1} \dots \delta_{c_k}^{n_k}$ be a multitwist in $\mcg(S)$. Then,  
\[
\ihat(\tau_C\cdot [a], \,[b]) = \ihat([a], \,[b]) + \sum_{i=1}^k n_k \cdot \ihat([a], \,[c_i])\cdot \ihat([c_i],\,[b]).    
\]
\end{lemma}

\section{Proof of Main Result}\label{secIntoTheProof}
Consider two braided multitwists
\begin{align*}
	&\tau_A = \delta_{a_1}^{n_1} \dots \delta_{a_k}^{n_k}, \\ 
	&\tau_B=\delta_{b_1}^{m_1}\dots\delta_{b_l}^{m_l}.
\end{align*}
Recall the subindices $A,B$ denote the sets $A=\{a_1,\dots, a_k\}$ and $B=\{b_1,\dots,b_l\}$. The first observation is that the braid relation implies
\begin{equation*}
	(\tau_A\tau_B) \tau_A (\tau_A\tau_B)^{-1}=\tau_B.
\end{equation*} In particular, for each curve $a_i$ there exists a curve  $b_j$ such that $\tau_A\tau_B \cdot a_i=b_j$ and $n_i=m_j$. Up to re-indexing $B$, we may assume  $\tau_A\tau_B\cdot a_i=b_i$ and $n_i=m_i$. Taking this into account,  we  write 
\begin{align*}
	&\tau_A = \delta_{a_1}^{n_1} \dots \delta_{a_k}^{n_k}, \\ 
	&\tau_B=\delta_{b_1}^{n_1}\dots\delta_{b_k}^{n_k}.
\end{align*}
\begin{remark}\label{rmk:accionenai}
	We emphasize that re-indexing  ensures  $\tau_A\tau_B \cdot a_i=b_i$. Also, we know $\tau_B\tau_A\cdot b_i \in A$, but \emph{a priori} $\tau_B\tau_A\cdot b_i$ could be different from $a_i$.
\end{remark}

As a  first step towards Theorem \ref{thm1}, the next lemma shows that braided multitwists decompose as a common part and two braided multitwists  sharing no curve.
\begin{lemma}\label{lemma:ReduceToNoCommonCurve}
	Let $S$ be an orientable surface. If $\tau_{A},\,\tau_{B}\in \mcg(S)$ are  two braided multitwists, then $ \tau_A = \tau_{A'}\tau_C$  and $ \tau_B = \tau_{B'}\tau_C$, where $\tau_C$ is a common multitwist and $\tau_{A'},\tau_{B'}$ are two braided multitwists sharing no curves. 
\end{lemma}
\begin{proof}
	 Let $C= A\cap B$ be the set of common curves of $A$ and $B$.  If $C$ is empty, then we  take $\tau_C=\text{id}$ and we are done.  If $C$ is non empty: 
	 
	 We may assume $\tau_A = \delta_{a_1}^{n_1} \dots \delta_{a_k}^{n_k}$, $\tau_B=\delta_{b_1}^{n_1}\dots\delta_{b_k}^{n_k}$ such that $\tau_A\tau_B\cdot a_i  = b_i$ and $C=\{a_l,\dots, a_k\}$. Note that every $ a_i\in C$ is a curve in $B$, therefore $a_i$ is disjoint from every  curve in $B$ and  so $b_i = \tau_A\tau_B\cdot a_i = \tau_A\cdot a_i = a_i$. That is, we may also write $C= \{b_l,\,\dots, \,b_{k}\}$. Now, we define 
	\begin{align*}
		&\tau_{A'}=\delta_{a_1}^{n_1} \dots \delta_{a_{l-1}}^{n_{l-1}},\\
		&\tau_{B'}=\delta_{b_{1}}^{n_{1}} \dots \delta_{b_{l-1}}^{n_{l-1}}
	\end{align*} 
	and \[ \tau_C = \delta_{a_l}^{n_l}\dots \delta_{a_k}^{n_k}. \]
	
	By definition $\tau_A = \tau_{A'}\tau_C$ and $\tau_B = \tau_{B'}\tau_C$, where $\tau_C$ is a common multitwist and $A', B'$ share no curves. To finish observe that  $\tau_{A'},\tau_{B'}$ are braided multitwists.  Indeed, the relation  $\tau_A\tau_B\tau_A = \tau_B\tau_A \tau_B$ implies 
	\[  \tau_{A'}\tau_C\tau_{B'}\tau_C\tau_{A'}\tau_C = \tau_{B'}\tau_C\tau_{A'}\tau_C \tau_{B'}\tau_C.\]
	Since $\tau_C$ commutes with both $\tau_{A'}$ and $\tau_{B'}$, the equation $\tau_{A'}\tau_{B'}\tau_{A'}= \tau_{B'}\tau_{A'}\tau_{B'}$ follows immediately. 
\end{proof}

It is worth noting that the previous lemma simplifies the proof of  Theorem \ref{thm1} to the case where  the multitwists share no curves. A preliminary observation on braided multitwists sharing no curves is the following:

\begin{lemma}\label{lem:IntersectAtLeastOne}
		Let $S$ be an orientable surface and consider two braided multitwists $\tau_{A},\,\tau_{B}\in \mcg(S)$. If $A$ and $B$ share no curves, then every  $a_i\in A$ intersects at least one curve in $B$.
\end{lemma}
\begin{proof}
	Take $\tau_A = \delta_{a_1}^{n_1} \dots \delta_{a_k}^{n_k}$ and $\tau_B=\delta_{b_1}^{n_1}\dots\delta_{b_k}^{n_k}$, so that $\tau_A\tau_B\cdot a_i  = b_i$. 
	
	Seeking a contradiction, suppose there is a curve $a_i\in A$  disjoint from every curve in $B$. Then, $b_i =\tau_A\tau_B\cdot a_i  = \tau_A\cdot a_i = a_i $ and $a_i\in B$. But this is not possible since $A$ and $B$ share no curves. 
\end{proof}

The goal of the next section is to understand how the curves in $\tau_A$ and $\tau_B$ may intersect each other. 

\subsection{Curves in braided multitwists}\label{subsec:clasificacion}

Let $\tau_A, \tau_{B}\in \mcg(S)$ be two braided multitwists. In this section we describe how a curve  $a\in A$ may intersect the curves in $B$.  This description will come as a consequence of Proposition \ref{prop:formulaEscondidaIvanov}.

As in Section \ref{secIntoTheProof}, we write $\tau_A = \delta_{a_1}^{n_1} \dots \delta_{a_k}^{n_k}$ and $\tau_B=\delta_{b_1}^{n_1}\dots\delta_{b_k}^{n_k}$, so that $\tau_A\tau_B\cdot a_i  = b_i$. Using Ivanov's formula, we deduce

\begin{align}
		\label{ecIvanov}i(a_i,b_i) & = i(a_i, \tau_A\tau_B\cdot a_i) 	 \\ 
	& = i(a_i, \tau_B\cdot a_i)\nonumber \\ 
	& = \sum_{j=1}^k (|n_j|i(b_j,a_i)-1)i(b_j,a_i)+X(a_i, \tau_B)\nonumber
\end{align}
In Table \ref{table} we list the possible \highlight{(non-negative) } values of $i(a_i,b_i)$, $X(a, \,\tau_B)$ and $|n_i|$ for which Equation (\ref{ecIvanov}) is satisfied. Notice each curve $a_i\in A$ is of a single type in Table \ref{table}.  We collect these facts in the following proposition:
\newpage

\begin{proposition}\label{propClasificacion}
	Let $S$ be an orientable surface and consider two braided multitwists $\tau_A,\tau_B\in \mcg(S)$. If
	\begin{align*}
		\tau_A = \delta_{a_1}^{n_1} \dots \delta_{a_k}^{n_k}, \;\;\;
		\tau_B=\delta_{b_1}^{n_1}\dots\delta_{b_k}^{n_k}
	\end{align*}
	and  $\tau_A\tau_B\cdot a_i  = b_i$. Then, every curve $a_i \in A$ is of Type 1, 2, 3, 4 or 5 (see Table \ref{table}). 
\end{proposition}

Notice Proposition \ref{propClasificacion} already provides some restrictions on the curves $A\cup B$. For instance, we immediately know that $i(a,b) \in \{0,1,2\}$ for any two curves $a,b\in A\cup B$.

\renewcommand{\arraystretch}{1.5}
\begin{table}
	\begin{center}
		\rowcolors{2}{lightgray}{}
		\begin{tabular}{c|c|c|c|p{7cm}} 
			\textbf{Type} & $i(a_i, \,b_i)$ & $|n_i|$ & $X(a_i,\,\tau_B)$ & \textbf{Intersecting curves}\\
			\hline
			$1$ & $0$ & - & $0$ & if $i(a_i,b_j) \neq 0$ \highlight{ and $i\ne j$}, then $i(a_i,b_j)=1$ and $|n_j|=1$. \\

			$2$ & $1$ & $2$ & $0$ & if $i(a_i,b_j) \neq 0$ \highlight{ and $i\ne j$}, then $i(a_i,b_j)=1$ and $|n_j|=1$. \\
			
			$3$ & $1$ & $1$ & $1$ & if $i(a_i,b_j) \neq 0$, then $i(a_i,b_j)=1$ and $|n_j|=1$. \\
			
			$4$ & $1$ & $1$ & $0$ & There exists $b_l$ with $i(a_i, \,b_l)=1$ and $|n_l|=2$. Any other $b_j\in B\setminus \{b_l\}$ with $i(a_i,b_j) \ne 0$, satisfies $i(a_i,b_j)=1$ and $|n_j|=1$.\\

			$5$ & $2$ & $1$ & $0$ & if $i(a_i,b_j) \neq 0$ and $i\ne j$, then $i(a_i,b_j)=1$ and$|n_j|=1$. \\
		\end{tabular}
	\end{center}
	\caption{Description of curves $a_i\in A$ in braided multitwists $\tau_A, \tau_B$.}
	\label{table}
\end{table}

\subsection{Reducing braided multitwists}\label{subsec:reduccion}

Lemma \ref{lemma:ReduceToNoCommonCurve} reduces the proof of Theorem \ref{thm1} to the case where braided multitwists $\tau_A,\tau_B$ share no curves. In this section we further simplify $\tau_A,\tau_B$ to the case where every curve in $A$  intersects at least two curves in $B$. To start, the following lemma describes curves $a\in A$ that intersect a single curve in $B$. 

\begin{lemma}\label{lemaDisjuntodetodoSalvoUna}
	Let  $S$ be an orientable surface and $\tau_A,\tau_B\in \mcg(S)$  two braided multitwists. If $a\in A$ intersects only one curve $b\in B$. Then, $b=\tau_A\tau_B\cdot a$, $i(a, b)=1$ and $|n|=1$, where $n$ is the power of $a$ in $\tau_A$. Moreover, $b$ is disjoint from every curve in  $A\setminus \{a\}$.
\end{lemma}
\begin{proof}
	We may assume $\tau_A = \delta_{a_1}^{n_1} \dots \delta_{a_k}^{n_k}$ and $\tau_B=\delta_{b_1}^{n_1}\dots\delta_{b_k}^{n_k}$, so that $\tau_A\tau_B\cdot a_i  = b_i$. Further assume $a=a_1$ and $n=n_1$.

    If $a_1$ intersects a single curve, then $X(a_1,\tau_B)>0$. By Proposition \ref{propClasificacion}  the curve $a_1$ is of type 3. It follows that $X(a_1, \tau_B)=1$, $|n_1|=1$ and $i(a_1,b_1)=1$. Since $a_1$ intersects a single curve $b\in B$, then $b=b_1$. This establishes the first part of the lemma.

    Recall that $\tau_A\tau_B\tau_A\cdot a_i = b_i$ and $\tau_A\tau_B\tau_A\cdot B = A$. Therefore, if $a_1$ intersects a single curve in $B$, then $b_1$ intersects a single curve in $A$.
\end{proof}

If the curve $a_i$ is under the conditions of the previous lemma, the pair $a_i,b_i$ can be `deleted' from the multitwists $\tau_A,\tau_B$. Indeed, consider 
\begin{align*}
	&\tau_{A'} = \delta_{a_1}^{n_1} \dots \delta_{a_{i-1}}^{n_{i-1}} \cdot \delta_{a_{i+1}}^{n_{i+1}}\dots\delta_{a_k}^{n_k},\\ 
	&\tau_{B'}=\delta_{b_1}^{n_1}\dots \delta_{b_{i-1}}^{n_{i-1}} \cdot \delta_{b_{i+1}}^{n_{i+1}} \dots\delta_{b_k}^{n_k}.
\end{align*}
Then, it is easy to verify that $\tau_{A'}\tau_{B'}\tau_{A'} = \tau_{B'}\tau_{A'}\tau_{B'}$.

The previous process simplifies the braided multitwists $\tau_A,\tau_B$ by eliminating the components that are clearly braided and do not interact with any other curves in $\tau_A,\tau_B$. By simplifying until there are no such pairs, we obtain \emph{reduced} multitwists:

\begin{definition}[Reduced multitwists]\label{def:sencillisimo}
    Let $\tau_A,\tau_B\in \mcg(S)$ be two braided multitwists. We will say $\tau_A,\tau_B$ are \emph{reduced} if every curve  $a\in A$ intersects at least two distinct curves in $B$. 
\end{definition}
\begin{remark}
	If $\tau_A,\tau_B$ are reduced multitwists, then $A, B$ have no common curves. 
\end{remark}

A key ingredient in the proof of Theorem \ref{thm1} is that reduced multitwists are trivial. We record this fact in the next proposition, and dedicate most of the next sections to proving it. 

\begin{proposition}\label{propSencillisimo}
	Let $S$ be an orientable surface. If $\tau_A,\tau_B$ are reduced multitwists, then $\tau_A=\tau_B=1\in\mcg(S)$.  
\end{proposition}

To prove Proposition \ref{propSencillisimo} we will need to understand the action of $\tau_A,\tau_B$  on the curves $A\cup B$.

\subsection{Action on curves and their homology classes}
Let $\tau_A = \delta_{a_1}^{n_1} \dots \delta_{a_k}^{n_k}$ and $\tau_B=\delta_{b_1}^{n_1}\dots\delta_{b_k}^{n_k}$ be two braided multitwists. We will consider the action of $f=\tau_B \tau_A \tau_B $ on the curves $A\cup B$ and the action on their homology classes. 

Recall that up to re-indexing $B$, we may assume $f\cdot a_i =  b_i$.  We will write $a_{i_1}$ for the image  $f\cdot b_i = a_{i_1}$ and, in general, we write $a_{i_k}$ for the image $f \cdot b_{i_{k-1}} = a_{i_k}$. The orbits of curves will be denoted by  \[\orb_f(a_i) = \{f^k\cdot a_i| \;k\in \integers\}.\]
 \highlight{
 The action of $f$ by conjugation on $\tau_A$ and $\tau_B$ permutes the curves $A\cup B$ and preserves the powers of the associated Dehn twists. In other words, any two curves $a,b\in A\cup B$ in the same $f$-orbit have associated Dehn twists with the same powers in $\tau_A, \tau_B$. As a consequence, two curves $a, a' \in A$ in the same $f$-orbit have the same type in Table \ref{table}.}

We will also consider the action of $f$ on the homology  classes of curves in   $A\cup B$. To achieve this, we need to assign orientations to the curves in $A\cup B$:  for each orbit $\orb_f(-)$ take a representative of the orbit $a\in \orb_f(-)$  and fix an arbitrary orientation on $a$. Then, any other curve $c\in \orb_f(a)$  in the orbit inherits an induced orientation. This orientation comes from considering the smallest $k\in \mathbb{N}$ such that $f^k\cdot a = c$. Since $a$ is oriented it induces an orientation on $f^k\cdot a$, which we consider as the orientation of $c$. 

The selected orientations allows us to talk about homology classes $[c]$ for curves  $c\in A\cup B$. The following diagram schematizes our choice of orientations:
\[
[a_i] \xrightarrow{f} [b_i]  \xrightarrow{f} [a_{i_1}] \xrightarrow{f} [b_{i_1}]  \xrightarrow{f}[a_{i_2}]  \xrightarrow{f}\dots \xrightarrow{f}[b_{i_m}],
\] where the orientation $[a_i]$ is arbitrary.
 
 \begin{remark}
 Notice there could be an $m\in \mathbb{N}$ such that  $f^m\cdot [a_i]=-[a_i].$
 \end{remark}

The next step is to show that the orbit of every curve $c\in A\cup B$ has  size two, i.e, $|\orb_f(c)|=2$.

\subsection{Orbits have size two}\label{sec:OrbitaTam2}

\begin{proposition}\label{proposicionOrbitasDos}
    Consider two reduced multitwists $\tau_A, \tau_B$ and let $f=\tau_A\tau_B\tau_A$. Then, $|\orb_f(c)|=2$ for every curve  $c\in A\cup B$.
\end{proposition}

To \highlight{prove } Proposition \ref{proposicionOrbitasDos} we treat each type of curve $a_i\in A$ separately (see Table \ref{table}). Thus, this proof is split among various lemmas and propositions below. 

\begin{lemma}\label{lemaOrbitasInt2}
    Consider two reduced multitwists $\tau_A, \tau_B$ and let $f=\tau_A\tau_B\tau_A$. If $a_i\in A$ is a curve of type 5, then  $|\orb_f(a_i)|=2$.
\end{lemma}
\begin{proof}
    Looking for a contradiction, assume $|\orb_f(a_i)|>2$. In that case, there exists $b\in \orb_f(a_i)$ such that $f\cdot b=a_i$ and $b\neq b_i$. Notice  \[i(b,\, a_i) = i(f\cdot b, \, f\cdot a_i)= i(a_i,\,b_i)=2.\]

    Thus, $a_i$ has intersection two with two distinct curves  $b, b_i\in B$. But this is not possible (see Table \ref{table} for type 5 curves). 
\end{proof}

Before continuing with other cases, we require the following lemma.

\begin{lemma}\label{lemaInterseccionOrbitas}
    Consider two reduced multitwists $\tau_A, \tau_B$ and let $f=\tau_A\tau_B\tau_A$. For any curves $a_i\in A$ and  $b\in B$, we have that \[\ihat([a_i],[b]) = - \ihat([a_i], \, f^2\cdot[b])\]. 

    In particular, if $\ihat([a_i],[b])=\pm1$ then $i(a_i, \,f^{2k}\cdot b)=1$ for every $k\in \mathbb{Z}$. 
\end{lemma}
In words, a curve $a_i\in A$ has the same algebraic intersection number (up to sign) with every curve in $\orb_f(b)\cap B$. 
\begin{proof}
	Let $a\in A$ be the curve $a=\tau_B\tau_A\cdot b$. Note  that  $\tau_B\tau_A \cdot[b] = \pm[a]$ and so it follows
    \[
        0  = \ihat([a_i], [a])  = \ihat([a_i], \pm \tau_B\tau_A \cdot [b])  = \ihat([a_i],  \tau_B\tau_A \cdot [b]).
    \]
    Now, by Equation \ref{homologicalActionMultitwist} we have  
    \begin{align*}
        0 & = \ihat([a_i],  \tau_B\tau_A \cdot [b]) \\ 
        &= \ihat\left([a_i], \tau_A \cdot [b]  + \sum_{b_k\in B} n_k \, \ihat(\tau_A\cdot [b], \,[b_k])\,[b_k] \right). \\ 
    \end{align*}
    From here, linearity yields
    \begin{equation}\label{ec1}
        \ihat([a_i], [b]) = -\sum_{b_k\in B} n_k \, \ihat(\tau_A\cdot [b], \,[b_k])\, \ihat([a_i], \, [b_k]).
    \end{equation}

    On the other hand, using the braid relation  
    \begin{align*}
        \ihat([a_i],\, f^2\cdot [b] ) & = \ihat([a_i], \, (\tau_A\tau_B\tau_A) (\tau_B\tau_A\tau_B)\cdot [b] ) \\ 
        &= \ihat ([a_i], \, \tau_B\tau_A\tau_B\tau_A \cdot [b]) \\ 
        &= \ihat([a_i], \, \tau_B^2 \tau_A\cdot [b]) \\ 
        &= \ihat\left([a_i],\, \tau_A\cdot[b] + 2\sum_{b_k\in B}n_k\,\ihat(\tau_A\cdot [b],[b_k])\,[b_k]\right) \\ 
        &= \ihat([a_i], \, [b]) + 2\sum_{b_k\in B} n_k \,\ihat(\tau_A\cdot [b],[b_k]) \,\ihat([a_i],[b_k]).
    \end{align*}

    Substituting Equation \ref{ec1} in the previous equality, we obtain 
    \[
        \ihat([a_i],\, f^2\cdot [b] ) = - \ihat([a_i], [b]).
    \]

    For the second part of the lemma, it is enough to observe that Proposition \ref{propClasificacion} implies $i(a, b)\in \{0,1,2\}$ for any two curves $a,b\in A \cup B$. Thus,  $\ihat(a_i,b)=\pm 1$ implies $i(a_i, b)=1$ and, since the algebraic intersection number only changes sign, we conclude $i(a_i, \,f^{2k}\cdot b)=1$.
\end{proof}

The next lemma proves Proposition \ref{proposicionOrbitasDos} for type 2 curves $a_i\in A$. 

\begin{lemma}\label{lemaOrbitasNi2}
    Consider two reduced multitwists $\tau_A, \tau_B$ and let $f=\tau_A\tau_B\tau_A$. If $a_i\in A$ is a type 2 curve, then  $|\orb_f(a_i)|=2$.
\end{lemma}
\begin{proof}
    Recall we denoted $b_i = f\cdot a_i$, $a_{i_1} = f\cdot b_i$ and  $b_{i_1} = f\cdot a_{i_1}$. To proceed by contradiction, we assume $|\orb_f(a_i)|>2$. In this case, $b_{i_1} \neq b_i$. 
    
    Since $a_i$ is a type 2 curve, then $i(a_i,\, b_i)=1$ and it follows from Lemma \ref{lemaInterseccionOrbitas} that $i(a_i, \, b_{i_1})=1$.  Moreover, $a_i, b_i$ and $b_{i_1}$ have exponent $|n_i|=2$ in $\tau_A,\tau_B$, since  $a_i$ is of type 2 and the three curves are in the same orbit. Summarizing, $a_i\in A$ is a curve that intersects two distinct curves $b_i, \,b_{i_1}\in B$ both having exponent $|n_i|=2$ in $\tau_B$. However, such  $a_i$ cannot be a curve in a braided multitwist (see Table \ref{table}).
\end{proof}

From here, the cases of Proposition \ref{proposicionOrbitasDos} become more subtle. The next result deals with type 1 curves.

\begin{lemma}\label{lemaOrbitasInt0}
    Consider two reduced multitwists $\tau_A, \tau_B$ and let $f=\tau_A\tau_B\tau_A$. If $a_i\in A$ is a type 1 curve, then $|\orb_f(a_i)|=2$.  
\end{lemma}
\begin{proof}
    Consider the curve $a_{i_1}=f^2 \cdot a_i$ and the set of curves $B_i = \{b \in B| \: i(a_i,b)\neq 0\}$. Given that $a_i$ is a type 1 curve, we know every  $b\in B_i$ has exponent $\pm 1$ in $\tau_B$ and satisfies $i(a_i, b)=1$ (see Table \ref{table}). Furthermore, $X(a_i, \tau_B)=0$ and it follows that  $|B_i|$ is even.
    
    Let $T_i$ be an open regular neighborhood of  $a_i \cup \bigcup_{b\in B_i}b$. Note the subsurface $T_i$ is homeomorphic to a torus with punctures. Since $X(a_i, \tau_B)=0$,  if two curves  $\overline{b}_j, \overline{b}_{j+1}\in B_i$ bound a once punctured annulus in $T_i$, then their powers in $\tau_B$ have opposite signs.   We represent $T_i$ as a planar torus in Figure \ref{figToroPlano0}, where we denote $\overline{b}_j$ the curves in $B_i$ and the exponents correspond to those in $\tau_B$.

    	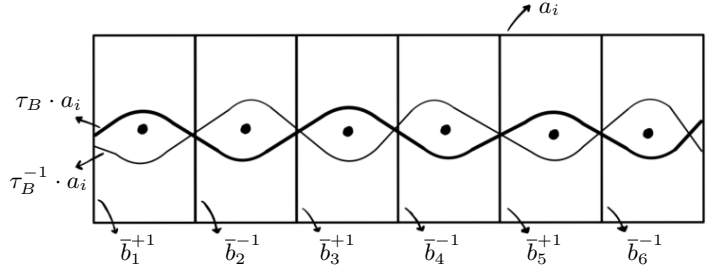
\begin{figure}[h]
    		\labellist
    		\small\hair 2pt
    		 \pinlabel {$a_i$} [ ] at 530 272
    		\pinlabel {$\tau_B\cdot a_i$} [ ] at 0 172
    		\pinlabel {$\tau_B^{-1}\cdot a_i $} [ ] at 0 90
    		\pinlabel {$\overline{b}_1^{+1}$} [ ] at 95 17
    		\pinlabel {$\overline{b}_2^{-1}$} [ ] at 206 17
    		\pinlabel {$\overline{b}_3^{+1}$} [ ] at 306 17
    		\pinlabel {$\overline{b}_4^{-1}$} [ ] at 417 17
    		\pinlabel {$\overline{b}_5^{+1}$} [ ] at 525 17
    		\pinlabel {$\overline{b}_6^{-1}$} [ ] at 632 17
    	\endlabellist
  		\centering
        \includegraphics[ width=0.7\linewidth, keepaspectratio]{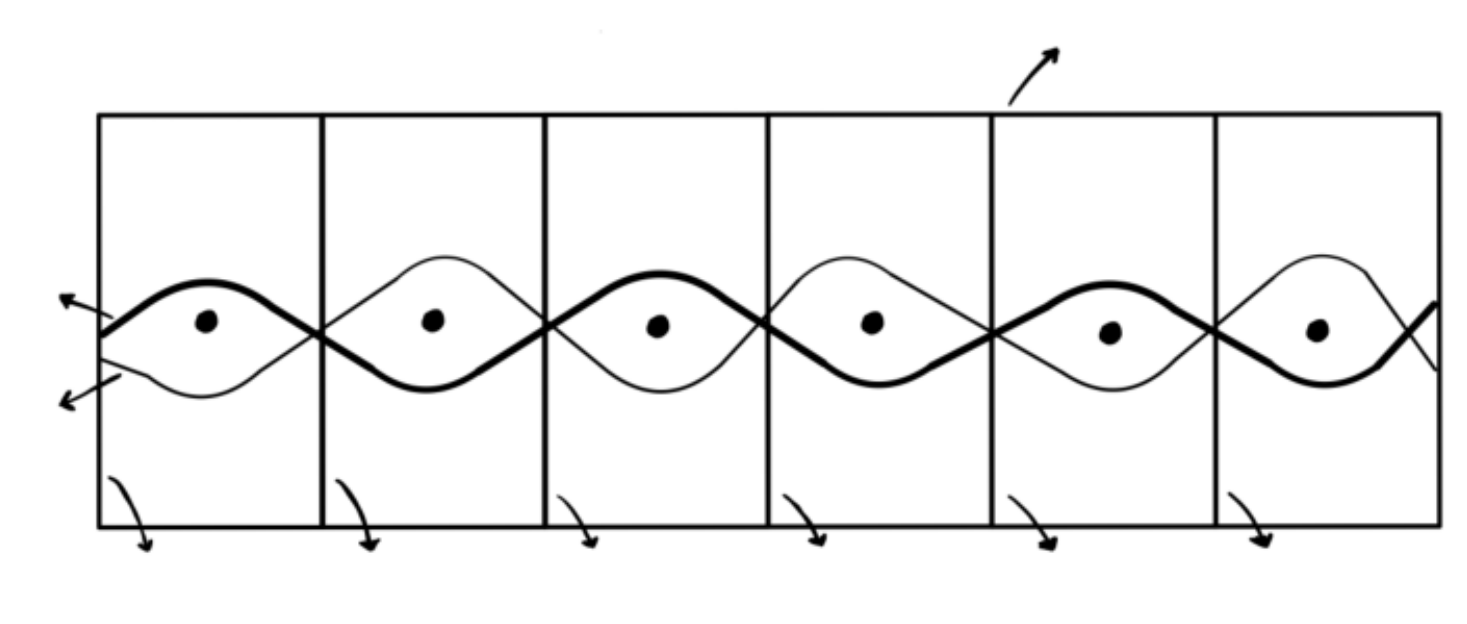}
        \caption{Planar torus $T_i$ for $i(a_i,b_i)=0$.}
        \label{figToroPlano0}
    \end{figure}

    Given that $a_i$ is a type 1 curve, we have that  $i(a_i, b)\in \{0,1\}$ for every $b\in B$. Since $a_{i_1}$ is in the orbit of $a_i$, then $a_{i_1}$ is also of type 1 and  $i(a_{i_1}, b)\in \{0,1\}$ for every $b\in B$.  It is now a consequence of Lemma \ref{lemaInterseccionOrbitas}  that $i(a_i, b)=i(a_{i_1}, b)$ for every $b\in B$. In particular, we know that  $i(a_{i_1}, b)= i(a_i,b)=i(a_i, b_i)=0$ for every  $b\in \orb_f(a_i)$. This implies 
    \begin{equation*}
        \begin{cases}
            i(a_{i_1},\, \tau_B\cdot a_i) = i(a_{i_1},\,b_i) = 0, \\ 
            i(a_{i_1}, \,\tau_B^{-1}\cdot a_i ) = i(a_{i}, \, \tau_B\cdot a_{i_1}) = i( a_i, b_{i_1})=0.
        \end{cases}
    \end{equation*}

    Notice the conditions ensure that  $a_{i_1}$ is either contained in $T_i$ or disjoint from $T_i$.  Indeed, any curve $c$ that intersects $T_i$ and is not contained in $T_i$ either satisfies  $i(c,\, \tau_B\cdot a_i)\ne 0$ or $i(c, \,\tau_B^{-1}\cdot a_i )\ne 0$ (see Figure \ref{figToroPlano0}). Now,  $i(a_{i_1}, \, \overline{b}_1)=i(a_i, \overline{b}_1)=1$ implies $a_{i_1}$ intersects $T_i$ and therefore $a_{i_1}$ is  contained in $T_i$.  

    To finish, we note that the only curve $a_{i_1}$ in $T_i$ satisfying $i(a_{i_1},\, \tau_B\cdot a_i) = i(a_{i_1}, \, \tau_B^{-1}\cdot a_i)=0$ is the curve $a_{i_1}=a_i$. So we conclude $|\orb_f(a_i)|=2$.
\end{proof}

Next lemma proves Proposition \ref{proposicionOrbitasDos} for type 4 curves.

\begin{lemma}\label{lemaOrbitasInt1X0}
    Consider two reduced multitwists $\tau_A, \tau_B$ and let $f=\tau_A\tau_B\tau_A$. If $a_i\in A$ is a type 4 curve, then $|\orb_f(a_i)|=2$.  
\end{lemma}
\begin{proof}
    The strategy of the proof is analogous to that of Lemma \ref{lemaOrbitasInt0}. Consider $a_{i_1}=f^2\cdot a_i$ and the set of curves  $B_i = \{b\in B | \; i(b,a_i) \ne 0\}$. Since $a_i$ is of type 4, it follows  that $i(b, a_i)=1$ for every $b\in B_i$, that $|B_i|$ is even and every curve in $B_i$ has power $\pm 1$  in $\tau_B$, except for a single curve which has power $\pm 2$ (see Table \ref{table}).  
    
    Let $T_i$  be an open regular neighbourhood of $a_i \cup\bigcup_{b\in B_i}b$. The above paragraph implies  $T_i$ is a  torus with punctures. Since $X(a_i, \, \tau_B)=0$, then  two curves  $\overline{b}_j, \overline{b}_k \in B_i$ bounding a single punctured annulus in $T_i$ have powers of opposite sign in $\tau_B$. In Figure \ref{figToroPlanoInt1X0} we represent $T_i$ as a planar torus, where we label $\overline{b}_j$ the curves in $B_i$ and the exponents correspond to those in $\tau_B$.
    
    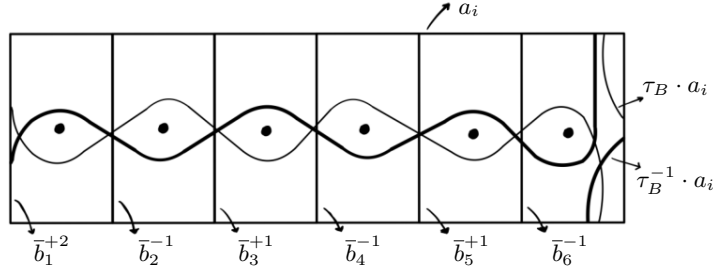
\begin{figure}[h]
    	\labellist
    	\small\hair 2pt
    	\pinlabel {$a_i$} [ ] at 492 268
    	\pinlabel {$\tau_B\cdot a_i$} [ ] at 712 184
    	\pinlabel {$\tau_B^{-1}\cdot a_i$} [ ] at 712 86
    	\pinlabel {$\overline{b}_1^{+2}$} [ ] at 48 11
    	\pinlabel {$\overline{b}_2^{-1}$} [ ] at 160 11
    	\pinlabel {$\overline{b}_3^{+1}$} [ ] at 266 11
    	\pinlabel {$\overline{b}_4^{-1}$} [ ] at 380 11
    	\pinlabel {$\overline{b}_5^{+1}$} [ ] at 495 11
    	\pinlabel {$\overline{b}_6^{-1}$} [ ] at 600 11
    	\endlabellist
    	\centering
        \includegraphics[width=0.7\linewidth]{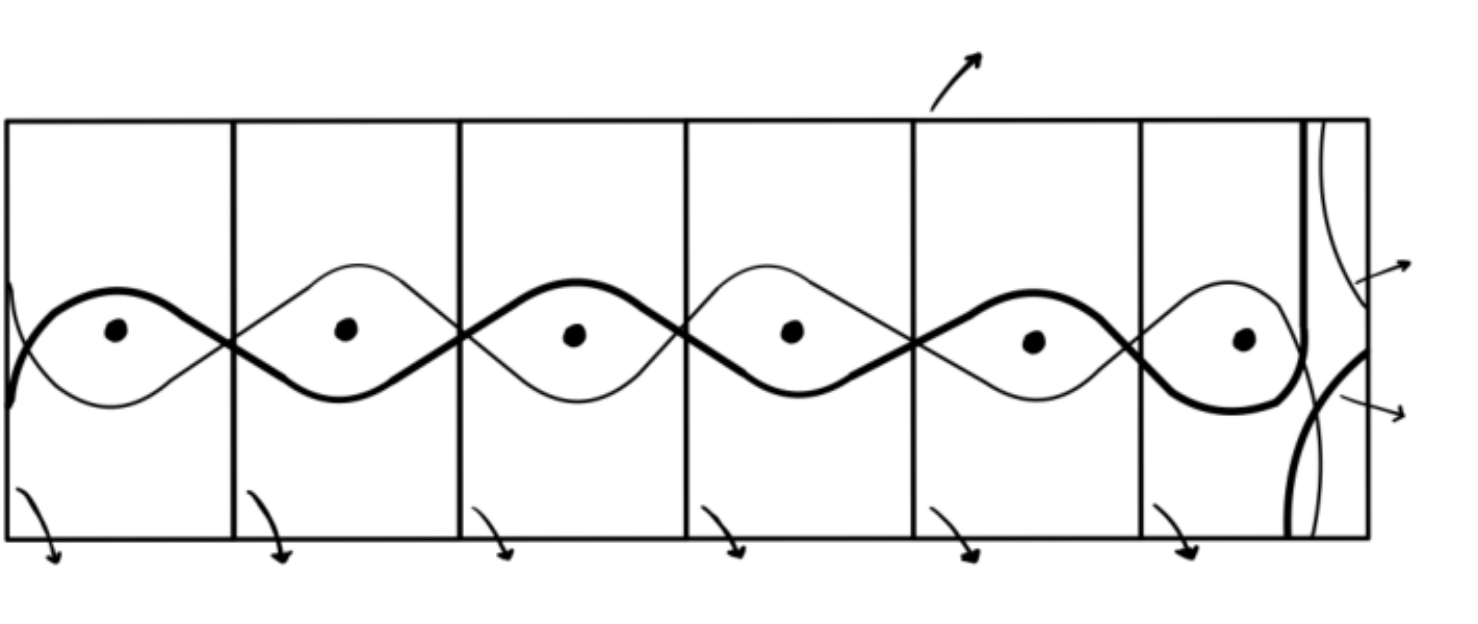}
        \caption{Planar torus $T_i$ for $i(a_i,b_i)=1$ and $X(a_i,\tau_B)=0$.}
        \label{figToroPlanoInt1X0}
      
    \end{figure}
    
    Given that $a_i$ and $a_{i_1}$ are type 4 curves, then $i(a_i, b), i(a_{i_1, b})\in \{0,1\}$  for every $b\in B$. Thus,  Lemma \ref{lemaInterseccionOrbitas} implies $i(a_i, b)=i(a_{i_1}, b)$ for every $b\in B$. It follows $a_{i_1}$ is a curve satisfying 

    \begin{equation*}
        \begin{cases}
            i(a_{i_1}, \, \overline{b}_j)=1, \\
            i(a_{i_1}, \, a_i)=0,\\
            i(a_{i_1},\, \tau_B\cdot a_i) = i(a_{i_1},\,b_i) = 1, \\ 
            i(a_{i_1}, \,\tau_B^{-1}\cdot a_i ) = i(a_{i}, \, \tau_B\cdot a_{i_1}) = i( a_i, b_{i_1})=1.
        \end{cases}
    \end{equation*}
    The only curve satisfying above equations is $a_{i_1}=a_i$ (see Figure \ref{figToroPlanoInt1X0}). One may check this fact by considering the point of intersection between $a_{i_1}$ and $\overline{b}_k$, where $\overline{b}_k\in B_i$ is the only curve in $B_i$ with exponent two. Then, by inspection, it is clear $a_{i_1}$ is contained in $T_i$. Readily, we obtain that $a_{i_1}$ and $a_i$ are homotopic, that is, $a_{i_1}=a_i$. Thus,  $|\orb_f(a_i)|=2$.
\end{proof}

Lastly, we consider the longer case of type  3 curves. We work through this case in the next three lemmas.

\begin{lemma}\label{lemaCansao}
    Consider two reduced multitwists $\tau_A, \tau_B$ and let $f=\tau_A\tau_B\tau_A$. If $a_i\in A$ is a type 3 curve and $|\orb_f(a_i)|>2$, then there exist $b_k\in B$ that intersects  $a_i$ and satisfies $|\orb_f(b_k)|=2$. Also, any other  $b\in B\setminus\{b_k\}$ intersecting $a_i$ satisfies $|\orb_f(b)| = 4$. 
\end{lemma}
\begin{proof}
    Let $B_i = \{b\in B|\:i(a_i,b)\ne 0\}$ and $T_i$ an open regular neighbourhood of the curves $a_i\cup \bigcup_{b\in B_i} b$. Given that $a_i$ is of type 3, we deduce $|B_i|$ is odd (see Table \ref{table}) and $T_i$ is a torus with $|B_i|$ punctures. 

    The intersection points of $a_i$ with curves in $B_i$ induce a cyclic order in $B_i$. Considering this order, we may write $\overline{b}_1, \overline{b}_2, \dots, \overline{b}_{|B_i|}$ to denote the curves in $B_i$. Since $X(a_i, \,\tau_B)=1$, we can arrange the subindices in $\overline{b}_j$ so that the powers in $\tau_B$ alternate sign. Even though the signs alternate, the curves $\overline{b}_1$ and $\overline{b}_{|B_i|}$ have the same sign since $|B_i|$ is odd.  In Figure  \ref{figEntRegX1} we represent $T_i$ as a planar torus. 
    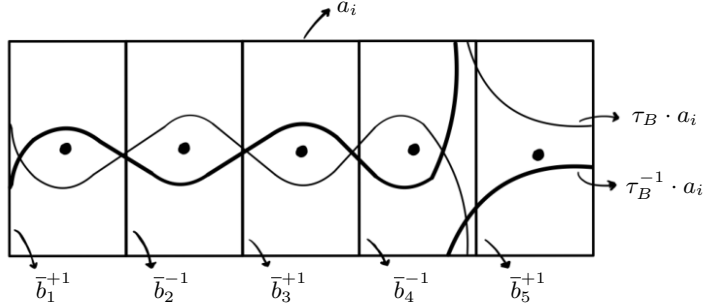
\begin{figure}[h]
    	\labellist
    	\small\hair 2pt
    	\pinlabel {$a_i$} [ ] at 365 295
    	\pinlabel {$\tau_B\cdot a_i$} [ ] at 693 182
    	\pinlabel {$\tau_B^{-1}\cdot a_i$} [ ] at 693 112
    	\pinlabel {$\overline{b}_1^{+1}$} [ ] at 62 7
    	\pinlabel {$\overline{b}_2^{-1}$} [ ] at 184 7
    	\pinlabel {$\overline{b}_3^{+1}$} [ ] at 305 7
    	\pinlabel {$\overline{b}_4^{-1}$} [ ] at 432 7
    	\pinlabel {$\overline{b}_5^{+1}$} [ ] at 552 7
    	\endlabellist
    	\centering
        \includegraphics[width=0.7\linewidth]{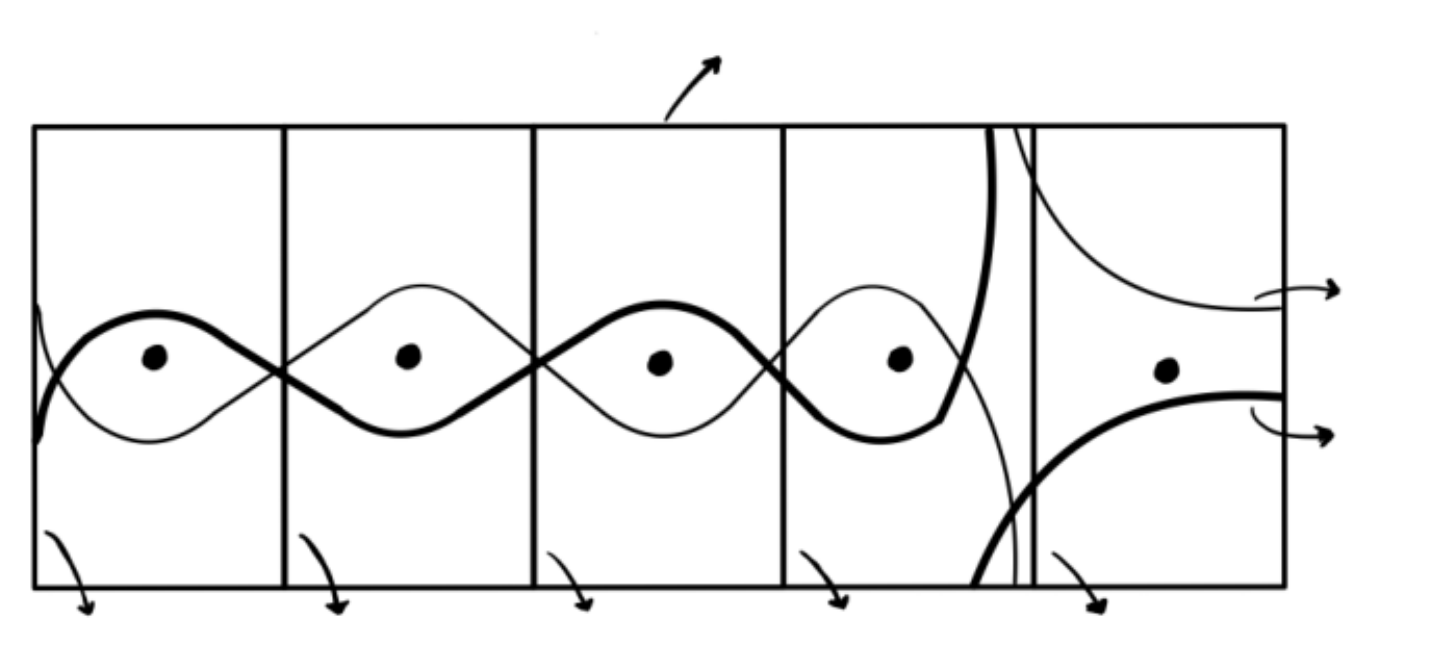}
        \caption{The planar torus $T_i$ for $i(a_i,b_i)=1$ and $X(a_i,\tau_B)=1.$}
        \label{figEntRegX1}
    \end{figure}

    Let $a_{i_1}=f^2\cdot a_i \in A$. By hypothesis $a_{i_1}\ne a_i$. As in previous lemmas, we may use Proposition \ref{propClasificacion}  and Lemma \ref{lemaInterseccionOrbitas} to check that  $a_{i_1}$ satisfies the  equations
    \[
    \begin{cases}
        i(a_{i_1}, a_i)=0, \\
        i(a_{i_1}, b)=i(a_i,b) &\forall b\in B, \\
        i(a_{i_1}, \tau_B\cdot a_i)=1, \\ 
        i( a_{i_1},\, \tau_B^{-1}\cdot a_i)=1. \\ 
    \end{cases}    
    \]
    From these conditions it follows that either $a_{i_1}=a_i$ or $a_{i_1}\cap T_i$ has a representative as in Figure \ref{figai1}. Namely,  $a_{i_1}$ has a representative in minimal position with curves in $A\cup B$, therefore disjoint from $a_i$ and intersecting once every $b\in B_i$; also,  $a_{i_1}\cap T_i$ is an arc with endpoints in a single puncture of $T_i$ and   $T_i\setminus (a_i \cup a_{i_1})$  is the union of an annulus and a $|B_i|-1$ punctured annulus. To check that $a_{i_1}=a_i$ or $a_{i_1}\cap T_i$ has a representative as in Figure \ref{figai1}, we may consider the point of intersection of $a_{i_1}$ with $\overline{b}_{|B_i|}$ (see Figure \ref{figEntRegX1}). Then, by inspection, we see that a curve $a_{i_1}$ satisfying above equations must be as claimed.
    
    \begin{figure}[h]
    	\labellist
    	\pinlabel {$\gamma_i$} [ ] at 437 429
    	\pinlabel {$\gamma_{i_1}$} [ ] at 549 353
    	\pinlabel {$a_{i_1}$} [ ] at 560 219
    	\endlabellist
    	\centering
        \includegraphics[width=0.6\linewidth]{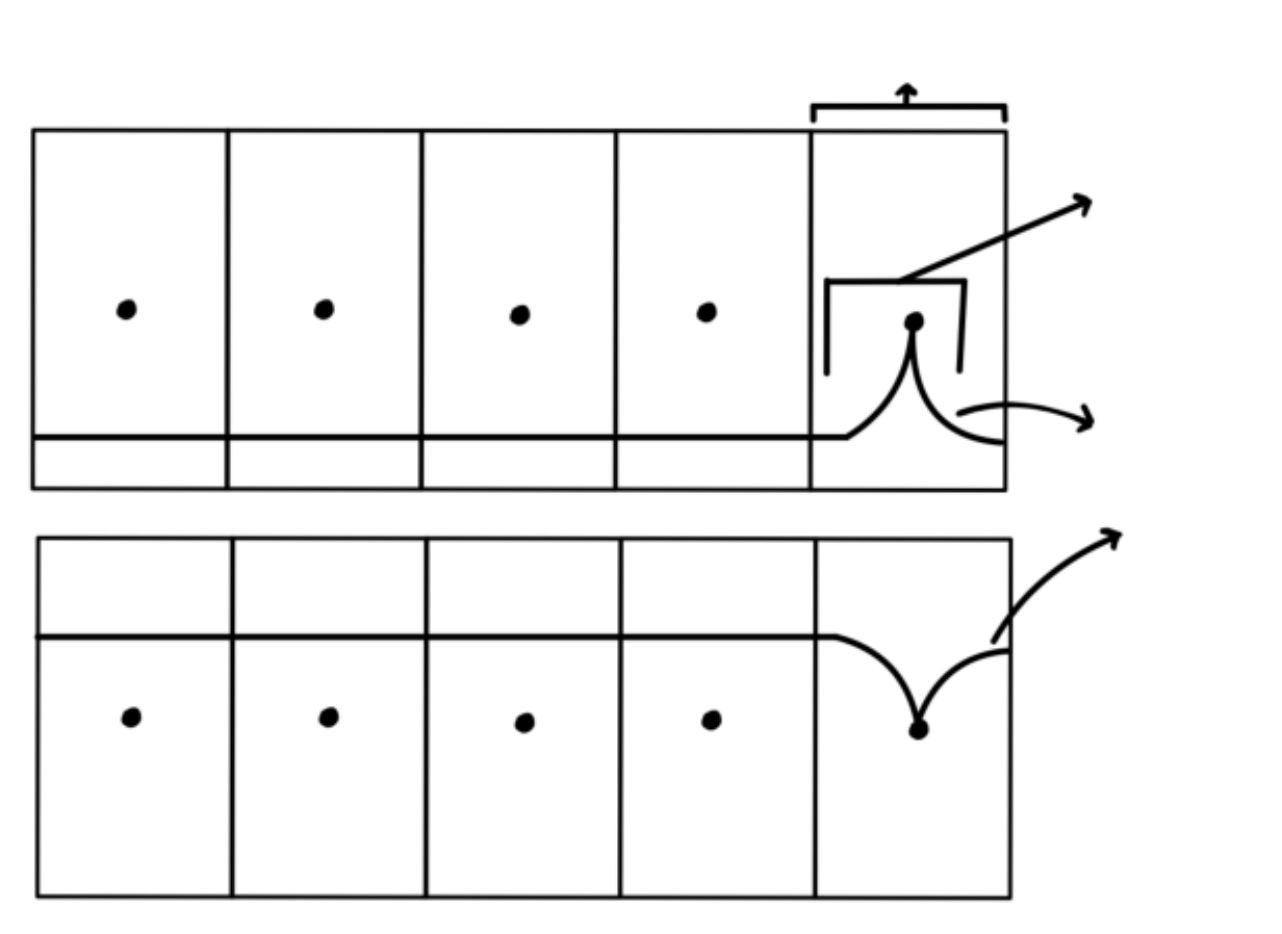}
        \caption{The two possible arcs for $a_{i_1}\cap T_i$.}
        \label{figai1}
    \end{figure}

    (We warn the reader that in the next paragraphs we silently use the Alexander Method \cite[Proposition 2.8]{farb_primer_2012}. We are taking representatives of $a_i$, $a_{i_1}$, every $b\in B_i$ and a representative of $f^2$, so that $f^2$ permutes the curves $b\in B_i$ and sends $a_i$ to $a_{i_1}$.)

    Now, denote  $\gamma_i$  the subarc of $a_i$ that goes from  $a_i\cap \overline{b}_{|B_i|}$  to $a_i \cap \overline{b}_{1}$ and is disjoint from the rest of $\overline{b}_j$'s. In the same style, define $\gamma_{i_1}$ to be the subarc of $a_{i_1}$ going from $a_{i_1}\cap \overline{b}_{|B_i|}$  to $a_{i_1} \cap \overline{b}_{1}$ and is disjoint from the rest of $\overline{b}_j$'s (see Figure \ref{figai1}). Observe that $f^2\cdot \gamma_i$ is a subarc of $a_{i_1}$ with interior  disjoint from the curves  $\overline{b}_j$'s. The endpoints of $f^2\cdot \gamma_i$ are points of intersection of $a_{i_1}$ with  $f^2\cdot \overline{b}_1$ and with $f^2\cdot \overline{b}_{|B_i|}$. Since $\overline{b}_1$ and $\overline{b}_{|B_i|}$ have the same power in $\tau_B$ and $f^2$ preserves the exponents, the curves $f^2\cdot \overline{b}_1$ and $f^2\cdot \overline{b}_{|B_i|}$ have the same exponent. But the only subarc of $a_{i_1}$ as above is $\gamma_{i_1}$, that is, $\gamma_{i_1}$ is the only subarc of $a_{i_1}$ between consecutive intersection points such that the intersecting curves have the same power. Thus,  $f^2\cdot \gamma_i = \gamma_{i_1}$ and either $f^2$ fixes $\overline{b}_1$ and $\overline{b}_{|B_i|}$, or it interchanges them.
    
    If $f^2$ fixes both $\overline{b}_1$ and $\overline{b}_{|B_i|}$, we have that $f^2$ fixes every $\overline{b}_j$. To see this, consider the subarc of $a_i$ from $\overline{b}_1$ to $\overline{b}_2$ (disjoint from every other $\overline{b}_j$). Naturally, $f^2$ must send this subarc to the subarc of $a_{i_1}$ from $\overline{b}_1$ to $\overline{b}_2$. Since $f^2$ fixes $\overline{b}_1$, then it also fixes $\overline{b}_2$. By repeating the argument on adjacent subarcs, we conclude $f^2$ fixes every $\overline{b}_j$. However, this implies  $|\orb_f(\overline{b}_j)|=2$ for every $\overline{b}_j\in B_i$ and, since $b_i\in B_i$, this contradicts the hypothesis $|\orb_f(a_i)|>2$.  Thus, $f^2$ cannot fix $\overline{b}_1$ and $\overline{b}_{|B_i|}$.
    
    If, on the contrary, $f^2$ interchanges $\overline{b}_1$ and $\overline{b}_{|B_i|}$, i.e, $f^2\cdot \overline{b}_1=\overline{b}_{|B_i|}$ and $f^2\cdot \overline{b}_{|B_i|} = \overline{b}_1$. Again,  using the previous argument on subarcs, it is an easy exercise to check that $f^2\cdot \overline{b}_j = \overline{b}_{|B_i|+1-j}$ and $f^2\cdot \overline{b}_{|B_i|+1-j}  = \overline{b}_j$. Therefore, $|\orb_f(b)|=4$ for every $b\in B_i \setminus \{\overline{b}_{\frac{|B_i|+1}{2}  }\}$  and $|\orb_f(\overline{b}_{\frac{|B_i|+1}{2}} )|=2$. The $b_k$ of the statement corresponds to  $b_k =\overline{b}_{\frac{|B_i|+1}{2}}$.
\end{proof}

\begin{lemma}\label{lemaCansao2}
    Consider two reduced multitwists $\tau_A, \tau_B$ and let $f=\tau_A\tau_B\tau_A$. If $a_i\in A$ is a type 3 curve and $|\orb_f(a_i)|>2$, then there exists a unique $b_k\in B$ that intersects $a_i$ and such that $|\orb_f(b_k)|=2$. Furthermore, let $a_k$ be the curve  $f\cdot a_k=b_k$. Then   $i(a_i, b) =i(a_k, b)$ for every  $b\in B$. 
\end{lemma}
\begin{proof}
    The first part of the lemma follows from Lemma \ref{lemaCansao}. We only need to see that $i(a_i, b) =i(a_k, b)$ for every $b\in B$.
    
    Denote $B_i =  \{b\in B|\: i(b,a_i)\neq 0\}$ and $B_k =  \{b\in B|\: i(b,a_k)\neq 0\}$. First, we will see  $B_i\subset B_k$:
    
    Let $b_j\in B_i$.  From Table \ref{table}, we know that $i(a_i, b_j) =1$. Now, using Lemma  \ref{lemaInterseccionOrbitas} we get $i(b_i,a_j)=1$ and thus   
    
    \begin{align}
        \label{eqwhatever}
        \pm 1 & = \ihat([b_i],[a_j])  \\ 
        &= \ihat(\tau_A \tau_B\cdot [a_i], [a_j]) \nonumber\\
        &= \ihat(\tau_B\cdot [a_i], [a_j]) \nonumber\\ 
        &=  \sum_{b_l\in B} n_l \,\ihat([a_i], [b_l])\,\ihat([b_l], [a_j])\nonumber\\ 
        &=  \sum_{b_l\in B_i\cap B_j} n_l \,\ihat([a_i], [b_l])\,\ihat([b_l], [a_j]).\nonumber
    \end{align}

    Now, considering Lemma \ref{lemaCansao}, we can partition the set $B_i$ as 
    \[
    B_i = B\cap\left(\orb_f(\overline{b}_2) \sqcup \orb_f(\overline{b}_3)\sqcup  \dots  \sqcup  \orb_f(\overline{b}_\frac{|B_i|-1}{2}) \sqcup \{b_k\}\right),  
    \]
    where each orbit  has two curves in $B$ (except for $\{b_k\}$). This induces a partition of $B_i\cap B_j$ as disjoint union of $\orb_f(\overline{b}_m) \cap B_j$ and $\{b_k\} \cap B_j$. We are looking to prove $b_k \in B_j$, so seeking a contradiction assume $b_k \not \in B_j$. By means of Lemma \ref{lemaInterseccionOrbitas}, we have
    \begin{align*}
        n_l \,\ihat([a_i], [b_l])\,\ihat([b_l], [a_j]) &= 
        n_l \,\ihat([a_i], f^2\cdot[b_l])\,\ihat(f^2\cdot[b_l], [a_j]) \\ 
        &=
        \pm n_l \,\ihat([a_i], [f^2\cdot b_l])\,\ihat([f^2\cdot b_l], [a_j]).
    \end{align*}
    Note that for $\overline{b}_m \ne b_k$ each orbit $B_j \cap \orb_f(\overline{b}_m)$  has two elements, which contribute either as $2\cdot n_l \,\ihat([a_i], [b_l])\,\ihat([b_l], [a_j])$  or as zero to the sum. Therefore, if $b_k\not \in B_j$ the sum 
    \[
        \sum_{b_l\in B_i\cap B_j} n_l \,\ihat([a_i], [b_l])\,\ihat([b_l], [a_j])    
    \]
    is even. But this clearly contradicts that it is equal to $\pm 1$ (see Equation \ref{eqwhatever}). Hence, $b_k \in B_j$ and even more the argument yields that $n_k\ihat( [a_i], \, [b_k]) \, \ihat([b_k], [a_j]) = \pm \ihat([b_k], [a_j])$ is odd. Given that $i(b_k, a_j)\in \{0,1,2\}$, it follows that $i(b_k, \, a_j)=1$ and by Lemma \ref{lemaInterseccionOrbitas} we have $i(a_k, b_j)=1$. We conclude  $1=i(a_i, b)=(a_k, b)$ for every $b\in B_i$ and therefore  $B_i \subset B_k$.

	To complete the lemma, we need to see  $B_k\subset B_i$. Pursuing a contradiction, assume there exists  $b_j\in B$ with $i(a_k, b_j)\ne 0$ and $i(a_i, b_j)=0$. First, observe that $i(a_i, b_i)=i(a_i, b_k)=1$ implies  $j\ne i$ and  $j\ne k$. Now,   $i(a_k, b_j)=1$ follows from  $j\ne k$  (see Table \ref{table}). Moreover,   $0 = i(a_i,\, b_j)=i(b_i, \, a_{j_1})$ and  by Lemma \ref{lemaInterseccionOrbitas} we have $\ihat([b_i], \, [a_j])=0$. After this, consider 
    \begin{align*}
        0 &=  \ihat([b_i], \, [a_j])\\
        &=\ihat(\tau_A\tau_B\cdot [a_i],\, [a_j]) \\ 
        &= \sum_{b_l\in B_i}n_l\,\ihat([a_i],[b_l])\,\ihat([b_l],[a_j])\\ 
        &= \sum_{b_l\in B_i\cap B_j}n_l\,\ihat([a_i],[b_l])\,\ihat([b_l],[a_j]).
    \end{align*}
    Replicating the argument above yields that the previous  sum is odd. Indeed, each orbit $\orb_f(\overline{b}_m) \cap B_j$ contributes an even number to the sum, plus  
    \[
        n_k\,\ihat([a_i],[b_k])\,\ihat([b_k],[a_j]) =\pm1.
    \]
    But the sum cannot be odd and equal to zero. Thus, no such $b_j$ can exist. In other words, $B_k=B_i$ and it follows that  $i(a_i, \,b)=i(a_k, \,b)$ for every $b\in B$. 
\end{proof}

Finally, we prove Proposition \ref{proposicionOrbitasDos} for type 3 curves.

\begin{lemma}\label{lemaOrbitasInt1X1}
    Consider two reduced multitwists $\tau_A, \tau_B$ and let $f=\tau_A\tau_B\tau_A$. If $a_i\in A$ is a type 3 curve, then $|\orb_f(a_i)|=2$.  
\end{lemma}
\begin{proof}
    Assume $|\orb_f(a_i)|>2$. By Lemma \ref{lemaCansao2}, there exist $b_k \in B_i$ with $|\orb_f(a_k)|=2$ and $i(a_k, b)=i(a_i, b)$ for every $b\in B$. Since $|B_i|=|B_k|$ is odd, we know from Table \ref{table} that $a_k$ is a type 3 curve and so  $i(a_k,\,b_k)=1$, $|n_k|=1$ and $X(a_k, \tau_B)=1$.
    
    Take any curve  $b_j \in B_k \setminus \{b_k\}$,  from Lemma \ref{lemaCansao} we have that $|\orb_f(b_j)|=4$. Notice that $a_j$ is a type 3 curve,  since lemmas \ref{lemaOrbitasInt2}, \ref{lemaOrbitasNi2}, \ref{lemaOrbitasInt0} and \ref{lemaOrbitasInt1X0} rule out any other type in Table \ref{table}. Now, considering Lemma \ref{lemaCansao2}, we deduce  $i(a_k, b)=i(a_j, b)$ for every $b\in B$. As a consequence, any two curves $b_j, b_l\in B_k=B_i$ satisfy $i(a_j, b)=i(a_l, b)$ for every $b\in B$. 

    We continue by considering $T_k$ an open regular neighbourhood of the curves $a_k\cup \bigcup_{b\in B_k}b$ (see Figure \ref{figtk1}). For any $b_j \in B_k$ consider the curve $a_j$ such that $f\cdot a_j =b_j$. By the previous paragraph, we know $a_j$ satisfies the equations 
    \[ 
        \begin{cases}
            i(a_j, a_k) = 0, \\ 
            i(a_j, b) = i(a_k, b) &\forall b\in B, \\
            i(a_j, \tau_B\cdot a_k)=1,\\
            i(a_j, \tau_B^{-1}\cdot a_k)=1.\\
        \end{cases}    
    \]
    Since $a_j$ satisfies the same conditions as $a_{i_1}$ in the proof of Lemma \ref{lemaCansao}, we may use the same argument to find a representative of $a_j$  disjoint from $a_k$ such that  $a_j\cap T_k$ is an arc with endpoints in a single puncture, and  $T_k\setminus (a_k \cup a_j)$  is the union of an annulus and a $|B_k|-1$ punctured annulus. Simpler, $a_j\cap T_k$ is as in Figure \ref{figtk1}.
    
    \begin{figure}[h]
    	\labellist
    	\small\hair 2pt
    	\pinlabel {$a_k$} [ ] at 585 300
    	\pinlabel {$a_j$} [ ] at 703 167
    	\pinlabel {$\overline{b}_1^{+1}$} [ ] at 52 15
    	\pinlabel {$\overline{b}_2^{-1}$} [ ] at 167 15
    	\pinlabel {$\overline{b}_3^{+1}$} [ ] at 290 15
    	\pinlabel {$\overline{b}_4^{-1}$} [ ] at 399 15
    	\pinlabel {$\overline{b}_5^{+1}$} [ ] at 522 15
    	\endlabellist
    	\centering
        \includegraphics[width=0.6\linewidth]{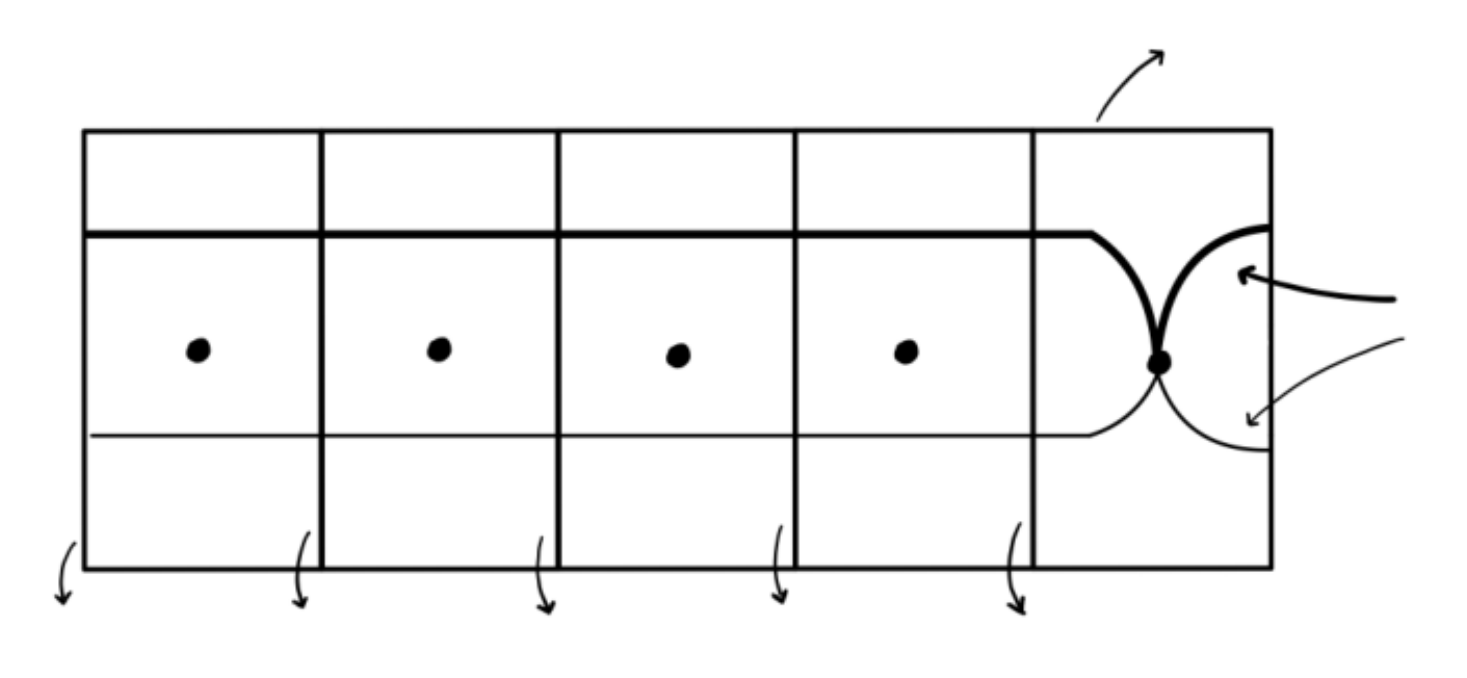}
        \caption{The planar torus $T_k$. The thick line and the thin line represent the two possible arcs for $a_{j}\cap T_i$.}
        \label{figtk1}
    \end{figure}
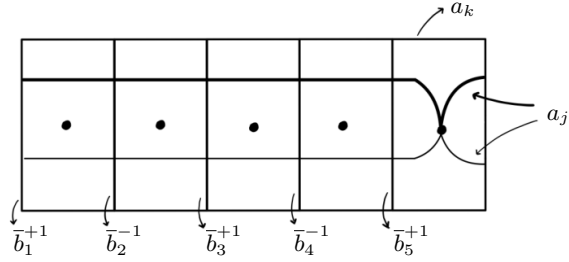

    We emphasize that $a_j\cap T_k$ is as in Figure \ref{figtk1} for any curve $a_j\in A$ with $f\cdot a_j = b_j\in B_k$. By taking appropriate representatives of each $a_j$, we may assume  that $\overline{b}_1$ and $\overline{b}_{|B_i|}$  induce the same cyclic order on the $a_j$'s (see Figure \ref{figtk2}). 
    
    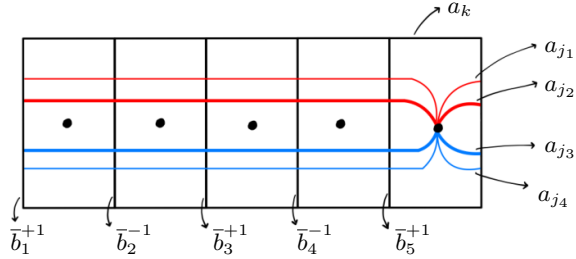
\begin{figure}[h]
    	\labellist
    	\small\hair 2pt
    	\pinlabel {$a_k$} [ ] at 560 298
    	\pinlabel {$a_{j_1}$} [ ] at 683 248
    	\pinlabel {$a_{j_2}$} [ ] at 683 200
    	\pinlabel {$a_{j_3}$} [ ] at 683 125
    	\pinlabel {$a_{j_4}$} [ ] at 675 65
    	\pinlabel {$\overline{b}_1^{+1}$} [ ] at 32 15
    	\pinlabel {$\overline{b}_2^{-1}$} [ ] at 162 15
    	\pinlabel {$\overline{b}_3^{+1}$} [ ] at 280 15
    	\pinlabel {$\overline{b}_4^{-1}$} [ ] at 383 15
    	\pinlabel {$\overline{b}_5^{+1}$} [ ] at 503 15
    	\endlabellist
    	\centering
        \includegraphics[width=0.6\linewidth]{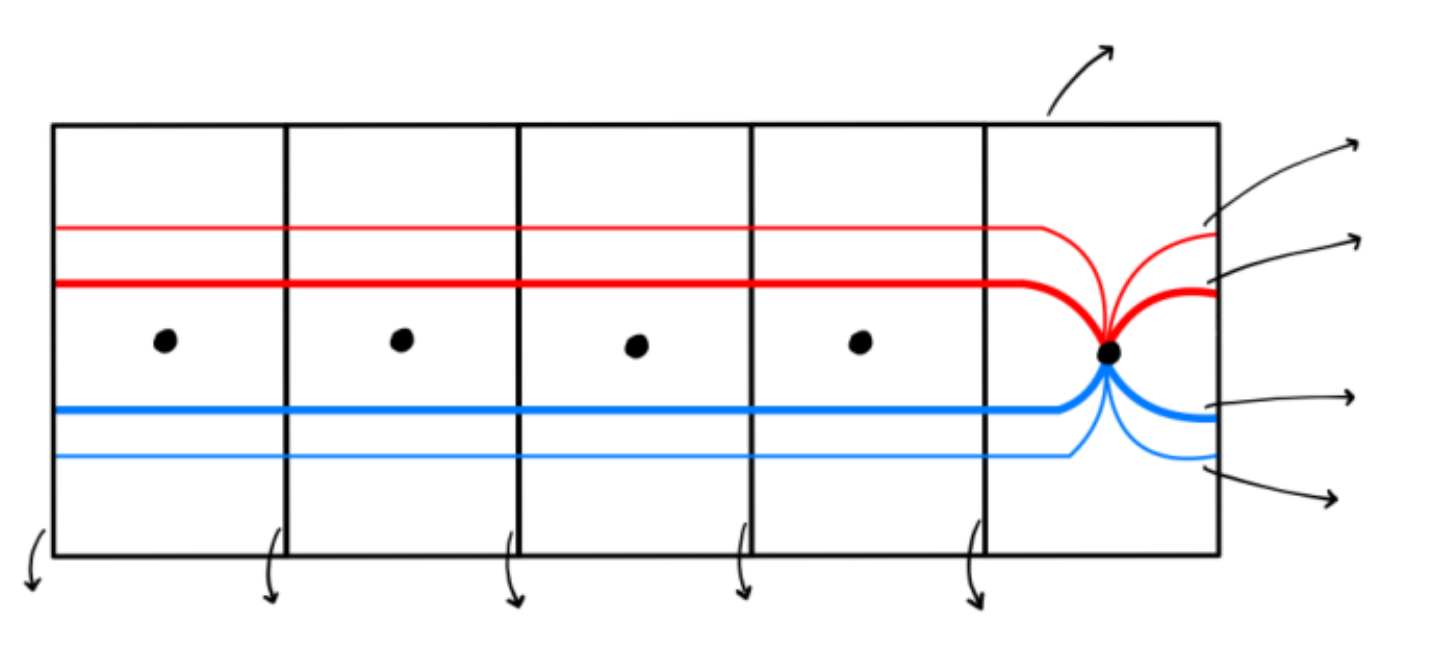}
        \caption{The planar torus $T_k$. The arcs $a_{j_l}\cap T_k$ for $a_{j_l}\in A$ are represented with different colours or thickness. }
        \label{figtk2}
    \end{figure}

    Now, consider $T$ the regular neighbourhood of $\bigcup_{b_j\in B_k}b_j \cup \bigcup_{b_j\in B_k}a_j$. The neighbourhood $T$ is homeomorphic to a  torus with punctures. We represent $T$ as a planar torus in Figure \ref{figtorusT}, where we denote $\overline{a}_j$ the curves that satisfy $f\cdot \overline{a}_j = \overline{b}_j$. To check that $T$ is a  torus with punctures, it is enough to consider a regular neighbourhood of the curves as depicted in Figure \ref{figtk2}.

    Recall that the intersection of $a_k$ with curves  $\overline{b}_j$ induces the cyclic order $\overline{b}_1, \dots,\,\overline{b}_{|B_i|}$. In the same fashion, the intersection of $b_k$ with curves $\overline{a}_j$ induces  the cyclic order $\overline{a}_1, \dots,\,\overline{a}_{|B_i|}$ (see Figure \ref{figtorusT}). This simply follows from the fact that $f\cdot \overline{a}_j=\overline{b}_j$ and $f\cdot a_k = b_k$. 

    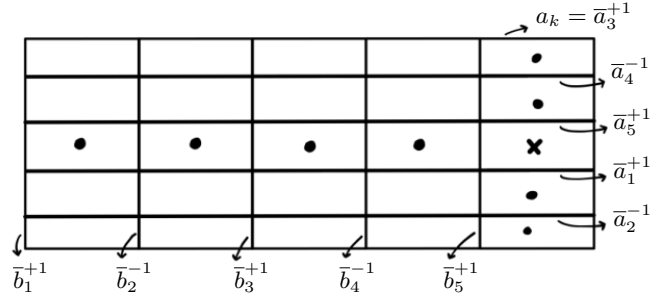
\begin{figure}[h]
    	\labellist
    	\small\hair 2pt
    	\pinlabel {$a_k=\overline{a}_3^{+1}$} [ ] at 567 269
    	\pinlabel {$\overline{a}_4^{-1}$} [ ] at 612 216
    	\pinlabel {$\overline{a}_5^{+1}$} [ ] at 614 167
    	\pinlabel {$\overline{a}_1^{+1}$} [ ] at 614 120
    	\pinlabel {$\overline{a}_2^{-1}$} [ ] at 614 74
    	\pinlabel {$\overline{b}_1^{+1}$} [ ] at 34 17
    	\pinlabel {$\overline{b}_2^{-1}$} [ ] at 134 17
    	\pinlabel {$\overline{b}_3^{+1}$} [ ] at 247 17
    	\pinlabel {$\overline{b}_4^{-1}$} [ ] at 349 17
    	\pinlabel {$\overline{b}_5^{+1}$} [ ] at 449 17
    	\endlabellist
    	\centering
        \includegraphics[width=0.7\linewidth]{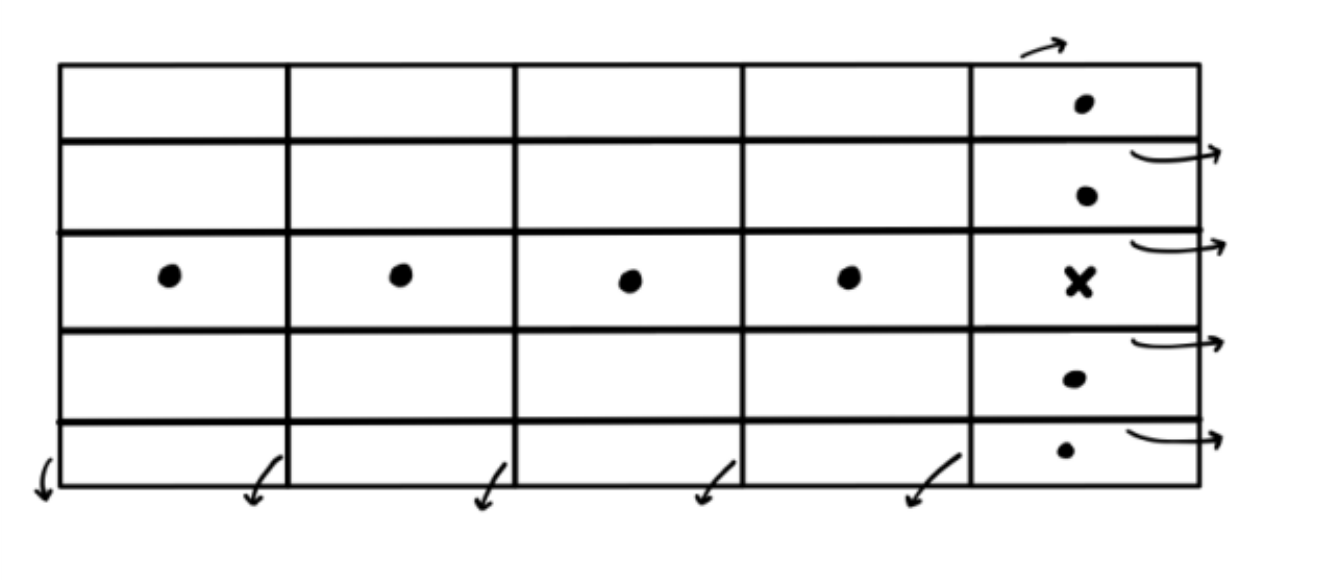}
        \caption{The planar torus $T$. The puncture marked with a cross might be a disk. }
        \label{figtorusT}
    \end{figure}

    Finally, one may directly check that $i(\tau_B\tau_A \cdot \overline{b}_{|B_i|}, \, \overline{a}_1)\ne 0$ by taking a representative of $\tau_B\tau_A \cdot \overline{b}_{|B_i|}$ in $T$ and using the bigon criterion. This clearly yields a contradiction, as  $\tau_B\tau_A \cdot \overline{b}_{|B_i|}\in A$ and any two curves in $A$ are disjoint. 
\end{proof}

Collecting the previous lemmas we conclude  $|\orb_f(a_i)|=2$ for any curve $a_i\in A$. This ends up the proof of Proposition  \ref{proposicionOrbitasDos}.

\begin{proof}[Proof of Proposition \ref{proposicionOrbitasDos}]
    Consider a curve $c\in A\cup B$. To prove $|\orb_f(c)|=2$, it is enough to prove it for  $c\in A$. By Proposition \ref{propClasificacion} the curve $c$ is of type 1, 2, 3, 4 or 5,  so it satisfies the conditions of either  Lemma  \ref{lemaOrbitasInt2}, \ref{lemaOrbitasNi2}, \ref{lemaOrbitasInt0}, \ref{lemaOrbitasInt1X0} or \ref{lemaOrbitasInt1X1}. In any case, the conclusion is  $|\orb_f(c)|=2$. 
\end{proof}

\subsection{Reduced multitwists are trivial}



We are close to proving Proposition \ref{propSencillisimo}. Before doing so, we devote the next lemmas to understand three concrete cases.

\begin{lemma}\label{lemaColofonNi2}
    Consider two reduced multitwists $\tau_A, \tau_B$ and let $f=\tau_A\tau_B\tau_A$. If $a_i\in A$ is a type 2 curve, then $a_i$ intersects at least three  distinct curves in $B$.
\end{lemma}
\begin{proof}
    Seeking a contradiction, assume $a_i$ intersects at most two curves in $B$. Since the multitwists are reduced, $a_i$ intersects exactly two curves in $B$.

    Take $b_i, b_j\in B$ to be the two curves that intersect $a_i$. We may assume without loss of generality that  $n_i=2$ and $n_j=-1$ (see Figure \ref{figNi2}). We have that

    \begin{align*}
        \pm 1 &= \ihat(\tau_A\tau_B \cdot [a_i], \, [a_j]) \\ 
        &= \ihat(\tau_B\cdot [a_i], [a_j]) \\ 
        &= \ihat([a_i] + 2\ihat([a_i],[b_i])\,[b_i]-\ihat([a_i],[b_j])\,[b_j], \, [a_j]) \\ 
        &=2\ihat([a_i],[b_i])\,\ihat([b_i], [a_j])-\ihat([a_i],[b_j])\,\ihat([b_j],[a_j])\\ 
        &= \pm 2 \pm \ihat([b_j], [a_j]). 
    \end{align*}
    Notice that for the previous equation to be satisfied, it is required that $i(b_j, a_j)=1$. 
    
    \begin{figure}[h]
    	\labellist
    	\small\hair 2pt
    	\pinlabel {$b_j^{-1}$} [ ] at 91 316
    	\pinlabel {$a_i^{+2}$} [ ] at 341 263
    	\pinlabel {$b_i^{+2}$} [ ] at 627 325
    	\endlabellist
    	\centering
        \includegraphics[width=0.5\linewidth]{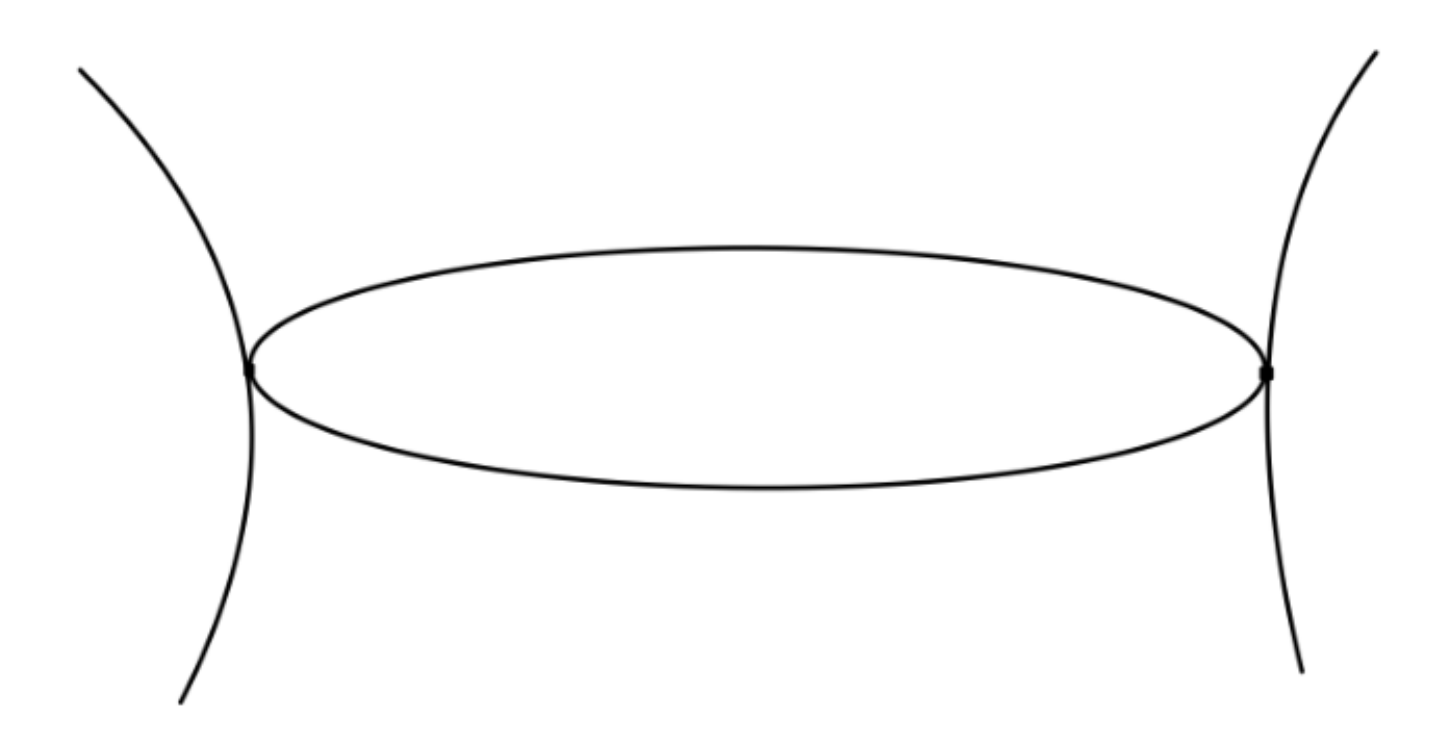}
        \caption{Regular neighbourhood of $i(a_i,b_i)=1$ with $|n_i|=2$ such that $a_i$ intersects only two curves in $B$.}
        \label{figNi2}
    \end{figure}

    To finish, consider $T_i$ a regular neighborhood of  $a_i\cup  b_i \cup b_j$. We represent $T_i$ as a planar torus in Figure  \ref{figToroYamecanse}. Now, we know  $a_j$ satisfies the equations 

    \begin{equation*}
        \begin{cases}
            i(a_j,\, b_i) = 1, \\
            i(a_j,\, a_i)=0,\\ 
            i(a_j ,\, \tau_B\cdot a_i ) = 1, \\ 
            i(a_j ,\,  \tau_B^{-1} \cdot a_i ) = 1.\\ 
        \end{cases}
    \end{equation*}
    These conditions imply that $a_j$ is a curve contained in $T_i$. But the only curve in $T_i$ satisfying above equations is $a_i$, so $a_j=a_i$. However, this is impossible since we assumed $b_j \ne b_i$.

    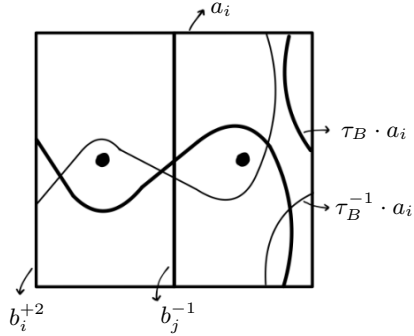
\begin{figure}[h]
    	\labellist
    	\small\hair 2pt
    	\pinlabel {$a_i$} [ ] at 256 360
    	\pinlabel {$\tau_B\cdot a_i$} [ ] at 430 220
    	\pinlabel {$\tau_B^{-1}\cdot a_i$} [ ] at 430 136
    	\pinlabel {$b_i^{+2}$} [ ] at 36 10
    	\pinlabel {$b_j^{-1}$} [ ] at 208 10
    	\endlabellist
    	\centering
        \includegraphics[width=0.4\linewidth]{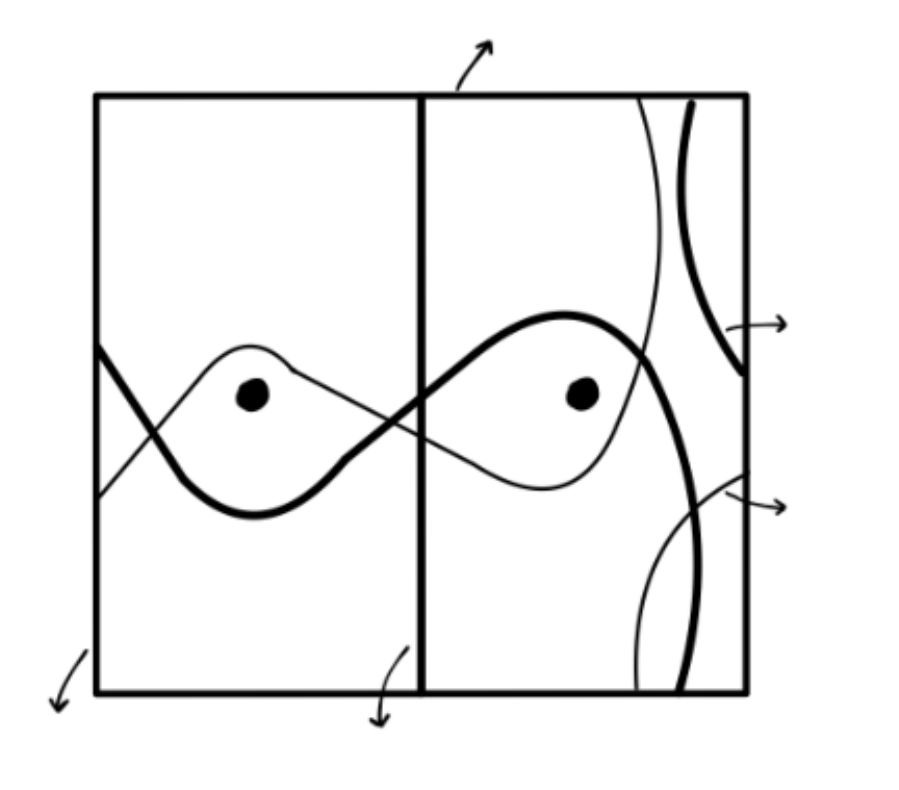}
        \caption{Planar torus $T_i$ with $i(a_i,b_i)=1$ and $|n_i|=2$, such that $a_i$ intersects only two curves in $B$.}
        \label{figToroYamecanse}
    \end{figure}

\end{proof}

\begin{lemma}\label{lemaColofonInt0}
    Consider two reduced multitwists $\tau_A, \tau_B$ and let $f=\tau_A\tau_B\tau_A$. If  $a_i\in A$ is a type 1 curve, then $a_i$ intersects at least three curves in $B$.
\end{lemma}
\begin{proof}
    Seeking a contradiction, assume $a_i$ intersects at most two curves in $B$. Since the multitwists are reduced, $a_i$ intersects exactly two curves in $B$.
    
    Take $b_j, b_k\in B$  to be the two curves that intersect $a_i$. Without loss of generality, we may assume the Dehn twists in $\tau_B$ have  exponents $n_i=n_k= 1$ and $n_j=-1$ (see Figure \ref{figEntornoIntPeq}). 
    
    \begin{figure}[h]
        \begin{center}
        		\labellist
        	\small\hair 2pt
        	\pinlabel {$b_j^{-1}$} [ ] at 91 316
        	\pinlabel {$a_i^{+1}$} [ ] at 341 265
        	\pinlabel {$b_k^{+1}$} [ ] at 627 325
        	\endlabellist
        	\centering
        	\includegraphics[width=0.5\linewidth]{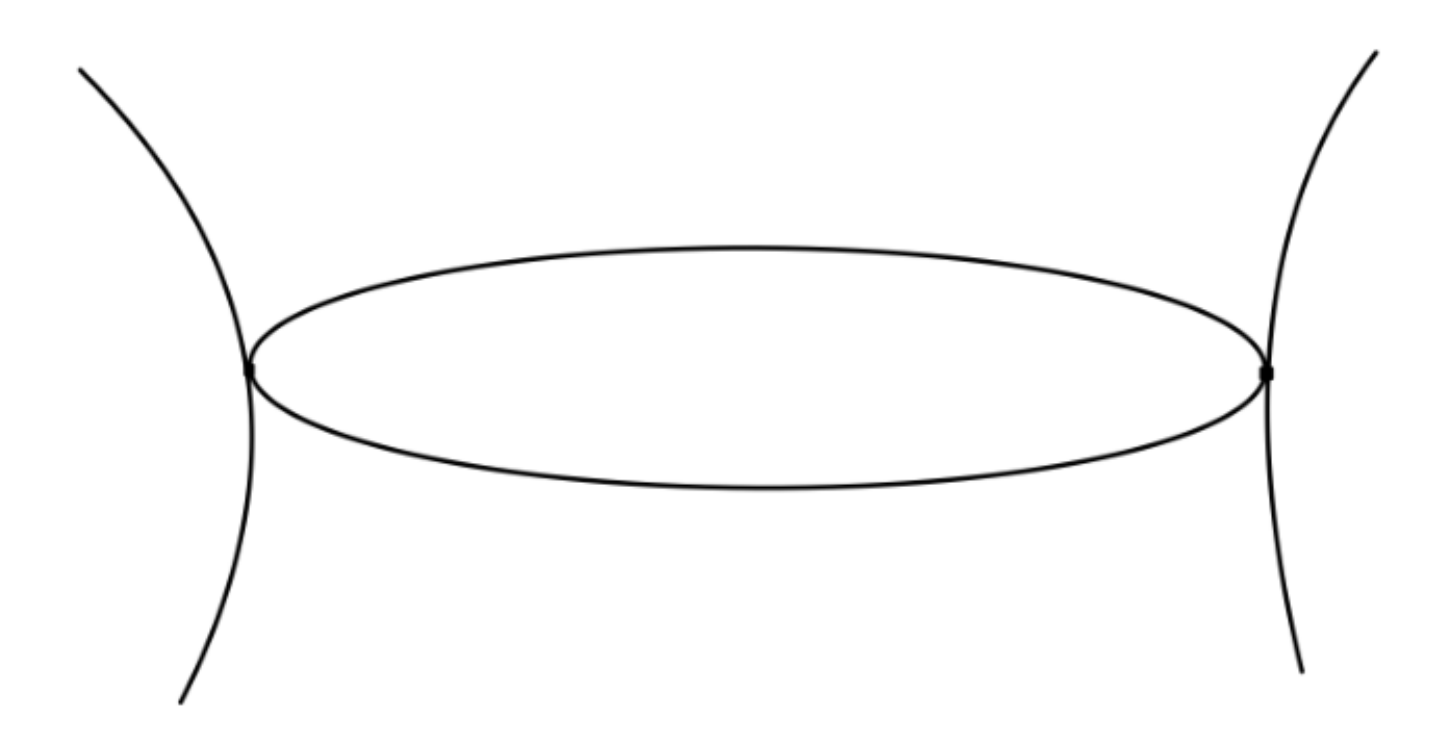}
        \caption{Regular neighbourhood of $a_i$ with $i(a_i,b_i)=0$ such that $a_i$ intersects exactly two curves in $B$.}
        \label{figEntornoIntPeq}
        \end{center}
    \end{figure}

    Let $T_i$ be a regular neighbourhood  of $a_i \cup b_j \cup b_k$. The subsurface $T_i$ is homeomorphic to a torus with two punctures. We represent $T_i$ as a planar torus in Figure \ref{figfalta}. Note that $a_j$ satisfies the equations 
    \begin{equation*}
        \begin{cases}
            i(a_j , \tau_B \cdot a_i)= 1, \\ 
            i(a_j, \tau_B^{-1} \cdot  a_i)= 1.
        \end{cases}
    \end{equation*}
    Any curve $a_j$ satisfying the conditions above is either $a_i$ or is a curve disjoint from $T_i$ (see Figure \ref{figfalta}). If $i(a_j,\, b_j)\ne 0$, then $a_j$ intersects $T_i$ and so $a_j=a_i$. However, this is not possible since $b_j\ne b_i$. Therefore, we can assume $i(a_j,b_j)=0$.

    \begin{figure}[h]
    	\labellist
    	\small\hair 2pt
    	\pinlabel {$a_i$} [ ] at 258 368
    	\pinlabel {$\tau_B^{-1}\cdot a_i$} [ ] at 440 230
    	\pinlabel {$\tau_B\cdot a_i$} [ ] at 435 140
    	\pinlabel {$b_k^{+1}$} [ ] at 49 17
    	\pinlabel {$b_j^{-1}$} [ ] at 216 17
    	\endlabellist
    	\centering
        \includegraphics[width=0.4\linewidth]{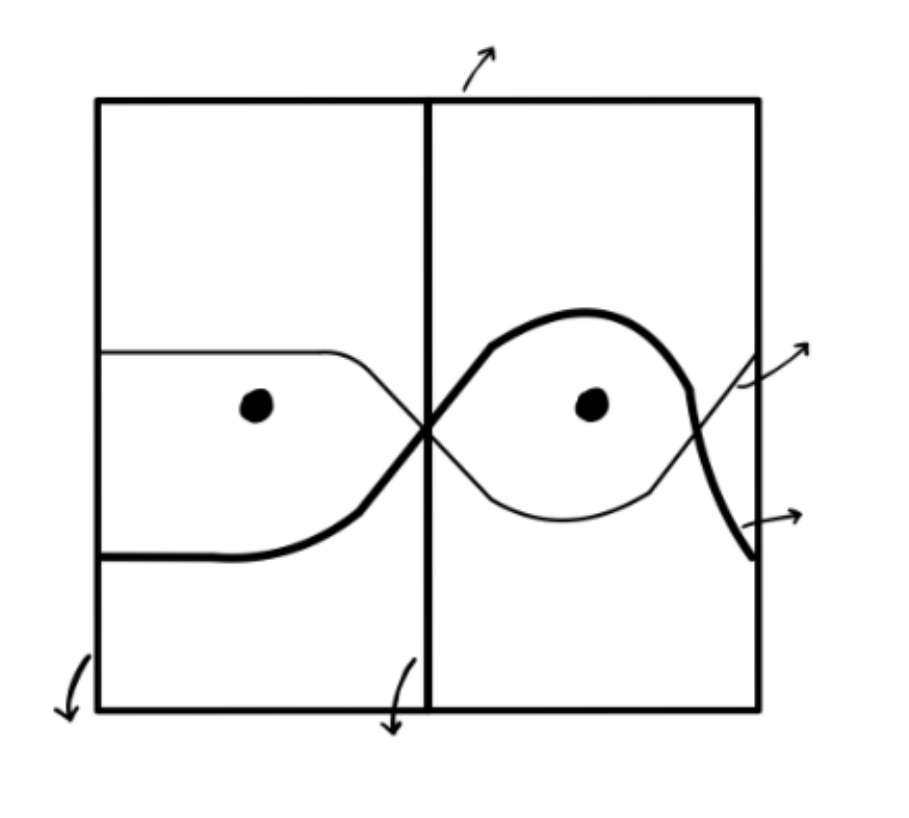}
        \caption{Regular neighbourhood of $i(a_i,b_i)=0$, where $a_i$ intersects two curves in $B$.}
        \label{figfalta}
    \end{figure}
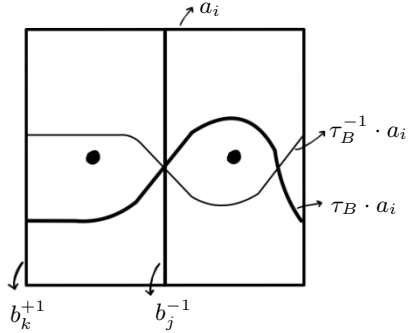

    Notice that
    \begin{align*}
        \pm 1 &= \ihat(\tau_B\cdot [a_i], \,[a_j]) \\ 
        &= \ihat([a_i] + \ihat([a_i],[b_k])\,[b_k] - \ihat([a_i], [b_j])\,[b_j], [a_j]) \\ 
        &= \ihat([a_i],[b_k])\,\ihat([b_k], [a_j])- \ihat([a_i],[b_j])\,\ihat([b_j], [a_j])\\ 
        &= \pm\ihat([b_k], [a_j]) \pm \ihat([b_j], [a_j]). \\ 
    \end{align*}
    If $i(a_j,b_j)=0$ then the previous equation implies $i(a_j,b_k)=1$. Again, it follows that $a_j$ is not disjoint from $T_i$ and so we have the contradiction $a_i=a_j$. 
\end{proof}

\begin{lemma}\label{lemaColofonInt2}
    Consider two reduced multitwists $\tau_A, \tau_B$ and let $f=\tau_A\tau_B\tau_A$. If $a_i\in A$ is a type 5 curve and $a_i$ intersects at most three curves in $B$, \highlight{then}, there exists $a_j\in A$ such that:
    \begin{itemize}
        \item either $a_j$ is a type 1 curve,
        \item or $i(a_j, b_j)=1$ (type 2, 3 or 4),
        \item or $i(a_j,b_j)=2$ and $a_j$ intersects at least four curves in $B$. 
    \end{itemize}
\end{lemma}
\begin{proof}
    Since $\tau_A, \tau_B$ are reduced, $a_i$ must intersect at least  two curves. Using that $i(a_i,\,b_i)=2$ and $X(a_i,\,\tau_B)=0$ (see Table \ref{table} for type 5 curves), we deduce $a_i$ intersects exactly three curves. 
    
    Let $b_i, b_j, b_k \in B$ be the curves intersecting  $a_i$. We may assume  that  $n_j=n_k=-1$ and $n_i=1$. Now, we have 
    \begin{align*}
         \pm 1 & = \ihat([a_j], [b_i]) \\ 
         &= \ihat([a_j], \tau_B\cdot [a_i]) \\ 
         &= \ihat([a_j], [a_i] + \ihat([a_i],[b_i])[b_i] - \ihat([a_i],[b_j])[b_j]- \ihat([a_i],[b_k])[b_k] )\\ 
         &=  \ihat([a_i],[b_i])\,\ihat([a_j],[b_i]) - \ihat([a_i],[b_j])\,\ihat([a_j],[b_j])- \ihat([a_i],[b_k])\,\ihat([a_j],[b_k] )\\ 
         &= \pm \ihat([a_i],[b_i]) \pm \ihat([a_j],[b_j]) \pm \ihat([a_j],[b_k] ).
    \end{align*}
    If $i(a_j,b_j)\in \{0,1\}$, then we already have the statement. Suppose that $i(a_j,b_j)=2$. 
    
    Using that $i(a_i,b_i)=2$, the previous equation implies that $i(a_j, b_k)=1$. Since $i(a_j,b_j)=2$, $X(a_j, \tau_B)=0$ and $n_k = n_j = -1$, it follows that  $a_j$ must intersect at least four curves in $B$ as otherwise $X(a_j, \tau_B)\ne0$. 
\end{proof}

Lastly, we are ready to prove Proposition \ref{propSencillisimo}.

\begin{proof}[Proof of Proposition \ref{propSencillisimo}]
    Assume  $\tau_A, \tau_B$ are non-trivial and let  $f=\tau_A\tau_B \tau_A$. For each curve $a_i\in A$ consider the subset  $B_i\subset B$ of curves intersecting $a_i$. Since $\tau_A$ is non-trivial, there exists an $a_i\in A$.

    If $i(a_i,\,b_i)\leq 1$ and $|B_i|\geq 3$, consider the graph $\mathcal{G}$  isomorphic to the union of curves  $a_i\cup \bigcup_{b\in B_i} b$. Recall that Proposition  \ref{proposicionOrbitasDos} implies that  $f^2\cdot c = c$ for every $c\in A\cup B$, so   $f^2$ defines an automorphism of the graph $\mathcal{G}$ by means of the Alexander method. Even more, $f^2$ fixes every vertex in  $\mathcal{G}$ and so it follows that $f^2$ preserves the orientation of $a_i$. In other words, $f^2\cdot[a_i]=[a_i]$. Notice there exists $b\in B_i$ with $i(a_i, b)=1$ and   
    \[\ihat([a_i], [b]) = \ihat(f^2\cdot[a_i], f^2\cdot [b]) = \ihat([a_i], f^2\cdot [b]),\]
    but this clearly contradicts Lemma \ref{lemaInterseccionOrbitas}. As a result, there cannot exist  $a_i\in A$ with $i(a_i,\,b_i)\leq 1$ and $|B_i|\geq 3$.

    Suppose $a_i\in A$ is a curve with $i(a_i,\,b_i)=2$ and $|B_i|\geq 4$, then the previous argument yields the same contradiction. For $i(a_i, b_i)=2$, we emphasize $|B_i|\geq 4$ is used to ensure $f^2$ fixes every vertex in $\mathcal{G}$.

     To finish, we list the rest of the cases according to Proposition \ref{propClasificacion}:
    
    \begin{itemize}
        \item If $a_i$ is a type 1 curve, then Lemma \ref{lemaColofonInt0} ensures $|B_i|\geq 3$ and we  already proved this leads to  contradiction.
        \item If $a_i$ is a type 2 curve, then Lemma  \ref{lemaColofonNi2} implies $|B_i|\geq 3$. Again, we  know this leads to contradiction.
        \item If $a_i$ is a type 4 curve, then Table \ref{table}  guarantees the existence of another curve  $a_j\in A$ of type 2. However, this puts us in the previous case, that has already been discarded.
        \item If $a_i$ is a type 3 curve, then $|B_i|$  is odd. Since the multitwists are reduced, it follows that  $|B_i|\geq 3$ and we already checked this case drives to  contradiction.
        \item At last,  if $a_i$ is a type 5 curve and $|B_i|\leq 3$, then Lemma  \ref{lemaColofonInt2} implies the existence of an $a_j$ of type 1, 2, 3, 4, or 5, and if $a_j$ is of type 5 then $|B_j|\geq 4$. However, we proved in the previous cases that such $a_j$ cannot exist. Thus, $a_i$ of type 5 and $|B_i|\leq 3$ also leads to contradiction. 
    \end{itemize}
    
    We have seen that the existence of $a_i\in A$ leads to a contradiction independently of its type. Therefore, $A$ is empty and so is $B$. In other words, $\tau_A=\tau_B = 1 \in \mcg(S)$. 
\end{proof}

With this in hand, we may quickly prove Theorem  \ref{thm1}. 

\begin{proof}[Proof of Theorem \ref{thm1}]

    Consider two braided multitwists $\tau_A, \tau_B$. Following Section  \ref{secIntoTheProof} we can write 
    \begin{align*}
    	&\tau_A = \delta_{a_1}^{n_1} \dots \delta_{a_k}^{n_k}, \\ 
    	&\tau_B=\delta_{b_1}^{n_1}\dots\delta_{b_k}^{n_k}.
    \end{align*}
    where $\tau_A\tau_B\tau_A \cdot a_i = b_i$. Additionally, we write $A=\{a_1,\dots, a_k\}$ and $B=\{b_1,\dots,b_k\}$. Using Lemma \ref{lemma:ReduceToNoCommonCurve}, we may assume that $A$ and $B$ share no curves and by Lemma \ref{lem:IntersectAtLeastOne} we may assume every curve   $a\in A$ intersects some curve $b\in B$. 

    Let $a_i\in A$ be a curve intersecting a single curve $b$ in $B$. In this case, Lemma \ref{lemaDisjuntodetodoSalvoUna} guarantees   $b=b_i=f\cdot a_i$, $i(a_i,b_i)=1$ and $|n_i|=1$. Moreover, we can decompose $\tau_A = \tau_{A'} \circ \delta_{a_i}^{n_i}$ and $\tau_B = \tau_{B'} \circ \delta_{b_i}^{n_i}$, so that $\tau_{A'}, \tau_{B'}$ are still braided and share no curves.  
  
    Repeating the above process until every curve  $a_i\in A'$  intersects at least two curves in $B'$, leaves us with two reduced multitwists $\tau_{A'},\tau_{B'}$. But by Proposition \ref{propSencillisimo}, we have that $\tau_{A'}=\tau_{B'}=1$. Thus, by Lemma \ref{lemaDisjuntodetodoSalvoUna} we conclude that the curves $a_1, \dots, a_k,b_1,\dots, b_k$  are pairwise disjoint except for $i(a_i,b_i)=1$ and that $n_i \in \{-1, 1\}$, just as we claimed in Theorem \ref{thm1}.
\end{proof}

\printbibliography

\end{document}